\documentclass[oneside,english]{amsart}
\usepackage[T1]{fontenc}
\usepackage[latin9]{inputenc}
\usepackage{color}
\usepackage{float}
\usepackage{amsthm}
\usepackage{amstext}
\usepackage{amssymb}

\makeatletter

\providecommand{\tabularnewline}{\\}

\numberwithin{equation}{section}
\numberwithin{figure}{section}
\theoremstyle{plain}
\newtheorem{thm}{\protect\theoremname}[section]
  \theoremstyle{plain}
  \newtheorem{lem}[thm]{\protect\lemmaname}
  \theoremstyle{definition}
  \newtheorem{defn}[thm]{\protect\definitionname}
  \theoremstyle{remark}
  \newtheorem{rem}[thm]{\protect\remarkname}
  \theoremstyle{plain}
  \newtheorem{prop}[thm]{\protect\propositionname}

\usepackage{geometry}
\large

\makeatother

\usepackage{babel}
  \providecommand{\definitionname}{Definition}
  \providecommand{\lemmaname}{Lemma}
  \providecommand{\propositionname}{Proposition}
  \providecommand{\remarkname}{Remark}
\providecommand{\theoremname}{Theorem}

\begin{document}

\title{Denominator Identity for Twisted Affine Lie Superalgebras}

\author{Shifra Reif}
\begin{abstract}
The study of denominator identities for Lie superalgebras was recently
developed by M. Gorelik, V.G. Kac, P.Moseneder Frajria, I. Musson,
P. Papi, M. Wakimoto and the author. In this paper we generalize these
identities to the twisted affine case, completing the result to all
affine Lie superalgebras. 
\end{abstract}

\thanks{Supported by the Minerva foundation with funding from the Federal
German Ministry for Education and Research.}

\maketitle

\section{Introduction}

The study of denominator identities for Lie superalgebras was initiated
by V.G. Kac and M. Wakimoto in \cite{KW}. These identities were shown
to have interesting applications to number theory (see \cite{KW,M,Z}),
vacuum modules and minimal $W$-algebras (see \cite{HR,KG}) and the
Howe duality for compact dual pairs (see \cite{FKP}).

The denominator identities for basic Lie superalgebras and (non-twisted)
affine Lie superalgebras with non-zero dual Coxeter number were formulated
and partially proven by V.G. Kac and M. Wakimoto \cite{KW}. Complete
proofs were given by M. Gorelik in \cite{G1,G2}. The denominator
identity for the strange $\hat{Q}$ series was conjectured in \cite{KW}
and proven by D. Zagier in \cite{Z} using analytic methods. The case
where the dual Coxeter number is zero was proven in \cite{GR}.

For a semisimple Lie algebra, the denominator identity follows by
applying the Weyl character formula to the trivial representation
and takes the form $e^{\rho}R=\sum_{w\in W}\left(\mbox{sgn}w\right)we^{\rho}$
where $R:=\prod_{\alpha\in\Delta^{+}}\left(1-e^{-\alpha}\right)$
is the denominator of $\mathfrak{g}$ corresponding to the positive
roots $\Delta^{+}$, $W$ the affine Weyl group and $\rho$ is chosen
such that $\left(\rho,\alpha\right)=\frac{1}{2}\left(\alpha,\alpha\right)$
for every simple root $\alpha$. Similarly, we have the identity $e^{\hat{\rho}}\hat{R}=\sum_{w\in\hat{W}}\left(\mbox{sgn}w\right)we^{\hat{\rho}}$
for the an affine algebra $\hat{\mathfrak{g}}$. The affine Weyl group
$\hat{W}$ has a decomposition $\hat{W}=T\rtimes W$ where $W$ corresponds
to one of the maximal finite subalgebra of $\hat{\mathfrak{g}}$ and
$T$ is an abelian group of translations. Since $\hat{\rho}-\rho$
is $W$-invariant and the elements of $T$ are of sign $1$, we can
rewrite the denominator identity as 
\[
e^{\hat{\rho}}\hat{R}=\sum_{t\in T}t\left(e^{\hat{\rho}}R\right)=\sum_{t\in T}\left(\mbox{sgn}t\right)t\left(e^{\hat{\rho}}R\right).
\]

In this paper we prove the denominator identity for twisted affine
Lie superalgebras by generalizing the latter formula. In \cite{vdL1,vdL2},
V. de Leur classified these algebras and showed that all symmetrizable
Kac Moody superalgebras of finite growth are either finite, affine
or twisted affine. The classification consists of the series: $A\left(2k-1,2l-1\right)^{\left(2\right)}$,
$A\left(2k,2l-1\right)^{\left(2\right)}$, $A\left(2k,2l\right)^{\left(4\right)}$,
$C\left(k+1\right)^{\left(2\right)}$ and $D\left(k+1,l\right)^{\left(2\right)}$
and the exceptional Lie superalgebra $G\left(3\right)^{\left(2\right)}.$
The description of the automorphisms and the root systems given in
\cite{vdL2} is summarized in Appendix \ref{sub:construction}.

Let $\hat{\mathfrak{g}}=\tilde{\mathfrak{g}}^{\left(m\right)}$ be
a twisted affine Lie superalgebras defined by an automorphism of order
$m$ and $\mathfrak{g}$ the algebra formed by the fixed points under
this automorphism. Let $\hat{R}$ and $R$ be the denominators of
$\hat{\mathfrak{g}}$ and $\mathfrak{g}$ (see formulas (\ref{eq:R}),
(\ref{eq:hat R})), respectively, and $h^{\vee}:=\left(\hat{\rho},\delta\right)$
the dual Coxeter number where $\delta$ is the minimal imaginary root
(see \ref{sub:affine defs}). For $\hat{\mathfrak{g}}\ne A\left(2k-1,2l-1\right)^{\left(2\right)}$,
$k\ge l$, we take $T'$ to be the translation group that corresponds
to the {}``larger'' part of the Weyl group when $h^{\vee}\ne0$
and to the {}``smaller'' part when $h^{\vee}=0$ (see \ref{sub:W}).
For $\hat{\mathfrak{g}}=A\left(2k-1,2l-1\right)^{\left(2\right)}$,
$k\ge l$, we extend the affine Weyl group by a diagram automorphism
and take $T'$ to be the translation group corresponding to this extension.
The sign function is extended by setting this diagram automorphism
to be of sign 1.

We prove the following theorem:
\begin{thm}
\label{thm:identity}Let $q:=e^{-\delta}.$ Then the following identity
holds:
\begin{equation}
e^{\hat{\rho}}\hat{R}=f\left(q\right)\cdot\sum_{t\in T'}\left(\mbox{\emph{sgn}}t\right)t\left(e^{\hat{\rho}}R\right)\label{eq:denom}
\end{equation}
where
\[
f\left(q\right)=\left\{ \begin{array}{cc}
1 & \mbox{when }h^{\vee}\ne0\\
\prod_{n=1}^{\infty}\left(1-q^{2n+1}\right)^{-2} & A\left(2k-1,2k-1\right)^{\left(2\right)}\\
\prod_{n=1}^{\infty}\left(1+q^{2n+1}\right)^{-1} & A\left(2k,2k\right)^{\left(4\right)}\\
\prod_{n=1}^{\infty}\left(1-q^{2n+1}\right) & D\left(k+1,k\right)^{\left(2\right)}.
\end{array}\right.
\]

\end{thm}
We extend the proofs of the non-twisted cases in \cite{G1,GR}. The
proof splits into two parts, namely the cases of dual Coxeter number
zero and non-zero. We use the action of the Casimir operator as well
as the denominator identity for basic Lie superalgebras. As $e^{\hat{\rho}}\hat{R}$
is independent of the choice of simple roots, we prove the identity
for convenient choices described in Appendix \ref{sub:Appendix description of simple roots}.

Acknowledgments. The paper is based on a part of the author's dissertation
done under the supervision of Prof. M. Gorelik. The author would like
to thank M. Gorelik for reading several drafts of this paper, suggesting
ways to improve and fix it. The author is also grateful for A. Joseph
for helpful conversations and to G. Binyamini, C. Hoyt and X. Lamprou
for discussing on the presentation of this paper.

\section{Preliminaries}

We introduce standard notations and elementary facts that are used
in the paper.

\subsection{Twisted affine Lie superalgebras.\label{sub:affine defs}}

Let $\tilde{\mathfrak{g}}$ be a basic simple  Lie superalgebra with
non-degenerate invariant bilinear form $\left(\cdot,\cdot\right)$
and  $\sigma$ an automorphism of finite order $m>1$. The eigenvalues
of $\sigma$ are of the form $e^{\frac{2\pi i}{m}k}$, $k\in\mathbb{Z}_{m}$
and hence $\tilde{\mathfrak{g}}$ admits the following $\mathbb{Z}_{m}$-grading:
\[
\tilde{\mathfrak{g}}=\bigoplus_{k\in\mathbb{Z}_{m}}\tilde{\mathfrak{g}}_{\left(k\right)},\quad\tilde{\mathfrak{g}}_{\left(k\right)}=\left\{ x\in\tilde{\mathfrak{g}}\mid\sigma\left(x\right)=e^{\frac{2\pi i}{m}k}\cdot x\right\} .
\]
The twisted affine Lie superalgebra is defined to be 
\[
\hat{\mathfrak{g}}:=\left(\bigoplus_{k\in\mathbb{Z}_{m}}\mathbb{C}t^{k}\otimes\tilde{\mathfrak{g}}_{\left(k\left(\mbox{mod}m\right)\right)}\right)\oplus\mathbb{C}K\oplus\mathbb{C}D
\]
with the relations
\[
\left[t^{i}\otimes g_{1},t^{j}\otimes g_{2}\right]=t^{i+j}\otimes\left[g_{1},g_{2}\right]+i\delta_{i,-j}\left(g_{1},g_{2}\right)K,\quad\left[\hat{\frak{g}},K\right]=0,\quad\left[D,t^{i}\otimes g\right]=it^{i}\otimes g.
\]
The fixed points of $\sigma$ form a maximal finite subalgebra of
$\hat{\mathfrak{g}}$ which we denote by $\mathfrak{g}$. 

Let $\mathfrak{h}$ be a Cartan subalgebra of $\mathfrak{g}\subset\tilde{\mathfrak{g}}$
and $\tilde{\mathfrak{g}}=\bigoplus_{\alpha\in\mathfrak{h}^{*}}\tilde{\mathfrak{g}}_{\alpha}$
a decomposition where 
\[
\tilde{\mathfrak{g}}_{\alpha}:=\left\{ g\in\tilde{\mathfrak{g}}\mid\left[h,g\right]=\alpha\left(h\right)\cdot g,\,\,\forall h\in\mathfrak{h}\right\} .
\]
For every $i=0,1$, $j=0,\ldots m-1$, let 
\[
\Delta_{\bar{i}}^{(j)}:=\left\{ \alpha\in\mathfrak{h}^{*}\mid\tilde{\mathfrak{g}}_{\alpha}\cap\tilde{\mathfrak{g}}_{(j)}\cap\tilde{\mathfrak{g}}_{\bar{i}}\ne\left\{ 0\right\} \right\} 
\]
be a multi-set such that the multiplicity of $\alpha$ is $\dim\tilde{\mathfrak{g}}_{\alpha}\cap\tilde{\mathfrak{g}}_{(j)}\cap\tilde{\mathfrak{g}}_{\bar{i}}$.
We get that the set of roots of $\mathfrak{g}$ is $\Delta=\Delta_{\bar{0}}\cup\Delta_{\bar{1}}$
where $\Delta_{\bar{i}}:=\Delta_{\bar{i}}^{(0)}$. The set of roots
$\hat{\Delta}$ of $\hat{\mathfrak{g}}$ takes the form $\hat{\Delta}=\hat{\Delta}_{\bar{0}}\cup\hat{\Delta}_{\bar{1}}$
where 
\[
\hat{\Delta}_{\bar{i}}=\left\{ l\delta+\alpha\mid\alpha\in\Delta_{\bar{i}}^{(j)},\, l=j\left(\mbox{mod}\, m\right)\mbox{ such that }l\delta+\alpha\ne0\right\} .
\]
Note that the sets $\Delta_{\bar{i}}^{(j)}$ may contain $0$, in
which case the imaginary roots $l\delta\in\hat{\Delta}$ can be either
odd or even.

Fix a set $\pi$ of simple roots of $\mathfrak{g}$. We take $\hat{\pi}=\left\{ \alpha_{0}:=\delta-\theta\right\} \cup\pi$
to be the set of simple roots of $\hat{\mathfrak{g}}$, where $\theta$
is the highest weight in $\Delta_{\bar{0}}^{\left(1\right)}\cup\Delta_{\bar{1}}^{\left(1\right)}$.
The root lattice of $\hat{\mathfrak{g}}$ is defined to be $\hat{Q}=\sum_{i=0}^{n}\mathbb{Z}\alpha_{i}$.
Let $\hat{Q}^{+}=\sum_{i=0}^{n}\mathbb{N}\alpha_{i}$. Define the
partial ordering on $\hat{\mathfrak{h}}^{*}$ by $\mu\ge\nu$ if $\mu-\nu\in\hat{Q}^{+}$.
We extend $\left(\cdot,\cdot\right)$ to $\hat{\mathfrak{g}}$ in
the standard way. Let $\rho\in\mathfrak{h}^{*}$ be such that $\left(\rho,\alpha\right)=\frac{1}{2}\left(\alpha,\alpha\right)$
for every $\alpha$ in $\pi$. Let $\Lambda_{0}\in\hat{\mathfrak{h}}^{*}$
be such that $\left(\Lambda_{0},\delta\right)=1$ and $\left(\Lambda_{0},\Lambda_{0}\right)=\left(\Lambda_{0},\mathfrak{h}\right)=0$.
Let $\hat{\rho}:=h^{\vee}\Lambda_{0}+\rho$.

 Let $\hat{\Delta}_{\bar{0}}^{+}$ and $\hat{\Delta}_{\bar{1}}^{+}$
be the positive even and odd roots of $\hat{\mathfrak{g}}$, respectively
and $\Delta_{\bar{0}}^{+}$ and $\Delta_{\bar{1}}^{+}$ the positive
even and odd roots of $\mathfrak{g}$, respectively. The denominators
of $\hat{\mathfrak{g}}$ and $\mathfrak{g}$ are defined to be 
\begin{eqnarray}
R:=\frac{R_{\bar{0}}}{R_{\bar{1}}}, &  & R_{\bar{0}}=\prod_{\alpha\in\Delta_{\bar{0}}^{+}}\left(1-e^{-\alpha}\right)\mbox{ and }R_{\bar{1}}=\prod_{\alpha\in\Delta_{\bar{1}}^{+}}\left(1+e^{-\alpha}\right)\label{eq:R}\\
\hat{R}:=\frac{\hat{R}_{\bar{0}}}{\hat{R}_{\bar{1}}}, &  & \hat{R}_{\bar{0}}=\prod_{\alpha\in\hat{\Delta}_{\bar{0}}^{+}}\left(1-e^{-\alpha}\right)\mbox{ and }\hat{R}_{\bar{1}}=\prod_{\alpha\in\hat{\Delta}_{\bar{1}}^{+}}\left(1+e^{-\alpha}\right).\label{eq:hat R}
\end{eqnarray}
Recall that $q=e^{-\delta}$. The affine denominator takes the form
\[
\hat{R}=R\cdot\prod_{n=1}^{\infty}\prod_{j=0}^{m-1}\frac{\prod_{\alpha\in\Delta_{\bar{0}}^{\left(j\right)}}\left(1-q^{mn+j}e^{-\alpha}\right)}{\prod_{\alpha\in\Delta_{\bar{1}}^{\left(j\right)}}\left(1+q^{mn+j}e^{-\alpha}\right)}
\]
 where the elements from $\Delta_{\bar{i}}^{\left(j\right)}$ are
taken with multiplicity.

\subsection{The Weyl group of a Lie superalgebra.}

One of the main tools that is used in the proof is the action of certain
subgroups of the Weyl group. We recall the definition of the Weyl
group for twisted affine Lie superalgebras and some facts about its
action.

An even root $\alpha$ is called \emph{principal} if in some base
$\Pi'$ obtained from $\pi$ by a sequence of odd reflections, either
$\alpha$ or $\frac{\alpha}{2}$ is simple (see \cite[5]{HS}). The
\emph{Weyl group} is defined to be the group generated by reflections
with respect to the principal roots. For basic simple Lie superalgebras,
it coincides with the group generated by reflections with respect
to the even roots. For twisted affine Lie superalgebra, the real even
roots can be identified with the real roots of a Kac-Moody algebra
and the Weyl group is generated by these roots (see \cite{S}).

We shall use the following lemma in the proof of Theorem \ref{thm:identity}
for the case $h^{\vee}\ne0$.
\begin{lem}
\emph{\label{lem:WeylGroupAction}(\cite[1.3.2]{G1}) }Let $\Pi_{+}$
be the set of principal roots satisfying $\left\langle \hat{\rho},\alpha^{\vee}\right\rangle \ge0$
for all $\alpha\in\Pi_{+}$, and $W_{+}$ the subgroup of the Weyl
group generated by the reflections $\left\{ s_{\alpha}\mid\alpha\in\Pi_{+}\right\} $.
Then

\emph{(i)} One has $\hat{\rho}-w\hat{\rho}\in\hat{Q}^{+}$ for all
$w\in W_{+}$.

\emph{(ii)} If $w=s_{\alpha_{i_{1}}}\cdot\ldots\cdot s_{\alpha_{i_{r}}}$is
reduced decomposition of $w\in W_{+}$, then 
\[
\mbox{ht}\left(\hat{\rho}-w\hat{\rho}\right)\ge\left|\left\{ j\mid\left\langle \hat{\rho},\alpha_{i_{j}}^{\vee}\right\rangle \ne0\right\} \right|.
\]

\end{lem}
(iii) The stabilizer of $\hat{\rho}$ in $W_{+}$ is generated by
the reflections $\left\{ s_{\alpha}\mid\alpha\in\Pi_{+}\mbox{ and }\left\langle \hat{\rho},\alpha^{\vee}\right\rangle =0\right\} $.

\subsection{The translation group $T'$ and the denominator identity for basic
Lie superalgebras.\label{sub:W}}

We define the subgroups of the Weyl group including $T'$ that are
used in the proof of Theorem \ref{thm:identity} and normalize the
bilinear form. We shall then recall the denominator identity for basic
Lie superalgebras.
\begin{defn}
Let $\Delta_{1}$ and $\Delta_{2}$ be irreducible finite root systems
of Lie algebras. We say that $\Delta_{1}$ is larger than $\Delta_{2}$
if its rank is larger, or if the ranks are equal and $\Delta_{2}$
can be embedded in $\Delta_{1}$. 
\end{defn}
For a basic simple Lie superalgebra (excluding $D\left(2,1,\alpha\right)$),
the set of even roots $\Delta_{\bar{0}}$ is a root system of a reductive
Lie algebra which is a product of at most two irreducible root systems.
Let $\Delta^{\#}$ be the larger among the two and $W^{\#}$ its Weyl
group. If $h^{\vee}\ne0$, we take $\Delta'$ (resp. $\Delta''$)
to be the larger (resp. smaller) simple root subsystem and conversely
otherwise. If the two simple parts are isomorphic, we pick an arbitrary
choice. When they are incomparable ($C_{n}$ and $B_{n}$), we shall
specify for each case. For $D\left(2,1,\alpha\right)$, $\Delta_{\bar{0}}=A_{1}\sqcup A_{1}\sqcup A_{1}$
and we take $\Delta'=\Delta^{\#}=A_{1}\sqcup A_{1}$, $\Delta''=A_{1}$.

We normalize the bilinear form such that it is positive definite on
$\Delta'$. Let $\hat{\Delta}'$ (resp. $\hat{\Delta}''$) be the
maximal affine root subsystem of $\hat{\Delta}_{\bar{0}}$ with containing
$\Delta'$ (resp. $\Delta''$). The intersection $\hat{\Delta}'^{+}:=\hat{\Delta}'\cap\hat{\Delta}^{+}$
is a choice of positive roots for $\hat{\Delta}'$. Let $\hat{\pi}'$
be the corresponding set of simple roots and choose $\hat{\rho}'\in\mbox{span}\left(\hat{\pi}'\cup\left\{ \Lambda_{0}\right\} \right)$
such that $\left(\hat{\rho}',\alpha\right)=\frac{1}{2}\left(\alpha,\alpha\right)$
for every $\alpha\in\hat{\pi}'$. The root lattice of $\hat{\Delta}'$
is $\hat{Q}':=\mathbb{Z}\hat{\Delta}'$. Let $\pi''$ be the set of
simple roots of $\Delta''$ with respect to $\Delta''^{+}:=\Delta''\cap\Delta^{+}$.

Let $W'$ and $\hat{W}'$ be the Weyl group of $\Delta'$ and $\hat{\Delta}'$,
respectively. Let $M'$ be the lattice generated by $\left\{ \frac{2n}{\left(\alpha,\alpha\right)}\alpha\mid n\delta-\alpha\in\hat{\Delta}',\alpha\in\mathfrak{h}^{*}\right\} $
and $T'$ the abelian group generated by the translations 
\[
\left\{ t_{\alpha}\left(\lambda\right)=\lambda+\left(\lambda,\delta\right)\alpha-\left(\left(\lambda,\alpha\right)+\frac{1}{2}\left(\alpha,\alpha\right)\left(\lambda,\delta\right)\right)\delta\mid\alpha\in M',\,\lambda\in\mathfrak{h}^{*}\right\} .
\]
Note that $t_{\frac{2n}{\left(\alpha,\alpha\right)}\alpha}=s_{n\delta-\alpha}s_{\alpha}$.
Unless $\hat{W}'$ is of type $A_{2n-1}^{\left(2\right)}$ and $W'$
is of type $D_{k}$, we have that $T'\subset\hat{W}'$ (since in these
cases $n\delta-\alpha\in\hat{\Delta}'$ implies that $\alpha\in\mathbb{Q}\Delta'$)
and $\hat{W}'=W'\ltimes T'$ (see \cite[6.5]{K}). 

When $\hat{W}'$ is of type $A_{2n-1}^{\left(2\right)}$ and $W'$
is of type $D_{k}$ (that is, when $\hat{\mathfrak{g}}=A\left(2k-1,2l-1\right)^{\left(2\right)}$,
$k\ge l$), we have that $\Delta'=\left\{ \varepsilon_{i}\pm\varepsilon_{j}\right\} $
and $M=\mbox{span}_{\mathbb{Z}}\left\{ \varepsilon_{1},\ldots,\varepsilon_{n}\right\} $.
Let $W_{C_{k}}$ and $\hat{W}_{C_{k}}$ be the Weyl groups of $C_{k}$
and $\hat{C}_{k}$, respectively. Note that $W_{C_{k}}=\left\langle W',s_{\varepsilon_{k}}\right\rangle $,
$\hat{W}_{C_{k}}=\left\langle \hat{W}',s_{\varepsilon_{k}}\right\rangle $
and $s_{\varepsilon_{k}}$ corresponds to a diagram automorphism of
$\Delta'$. One has $\hat{W}_{C_{k}}=T'\rtimes W_{C_{k}}$. We can
check that the sign function can be extended from $\hat{W}'$ to $\hat{W}_{C_{k}}$
by setting $\mbox{sgn}s_{\varepsilon_{k}}=1$ (that is, $\mbox{sgn}t_{\varepsilon_{k}}=-1$).
The bilinear form $\left(\cdot,\cdot\right)$ is invariant under diagram
automorphisms and thus under $\hat{W}_{C_{k}}$. A similar idea for
extending the Weyl group was used in \cite[7]{FKP}.

We take $W''$ and $\hat{W}''$ to be the Weyl groups of $\Delta''$
and $\hat{\Delta}''$, respectively. Define $M''$ as $M'$ (replacing
$\hat{\Delta}'$ with $\hat{\Delta}''$) and $T''$ as we defined
$T'$ (replacing $M'$ with $M''$). 

The \emph{defect} of $\mathfrak{g}$ is the dimension of a maximal
isotropic subspace of $\mathfrak{h}_{\mathbb{R}}^{*}:=\sum_{\alpha\in\Delta}\mathbb{R}\alpha$.
A subset $S\subset\Delta_{\bar{1}}^{+}$ is called \emph{isotropic}
if it spans an isotropic subspace and \emph{maximal isotropic} if
moreover $\left|S\right|$ is equal to the defect. Given a maximal
isotropic set $S$, there exists a choice of simple roots $\pi$ such
that $S\subset\pi$ (\cite[2.2]{KW}). Thus, for the rest of the paper
we shall assume that our choice of simple roots contains a maximal
isotropic set. 

Within this setup, recall that the denominator identity for basic
Lie superalgebras, proven in \cite{KW,G2}, takes the following form:

\begin{equation}
e^{\rho}R=\sum_{w\in W^{\#}}\left(\mbox{sgn}w\right)w\left(\frac{e^{\rho}}{\prod_{\beta\in S}\left(1+e^{-\beta}\right)}\right).\label{eq:fin dim denom}
\end{equation}

\subsection{Algebras of formal power series.\label{sub:formal power series}}

We recall the definition of the algebra of formal power series in
which the equality of Theorem \ref{thm:identity} holds (see also
\cite[1.4]{G1}).
\begin{defn}
Let $\mathcal{R}$ be the $\mathbb{Q}$-vector space spanned by the
sums of the form $\sum_{\nu\in\hat{Q}^{+}}b_{\nu}e^{\lambda-\nu}$
where $\lambda\in\hat{\mathfrak{h}}^{*}$ and $b_{\nu}\in\mathbb{Q}$.
For $Y:=\sum_{\nu\in\hat{\mathfrak{h}}^{*}}b_{\nu}e^{\nu}\in\mathcal{R}$,
we define the support of $Y$ to be 
\[
\mbox{supp}\left(Y\right):=\left\{ \nu\in\hat{\mathfrak{h}}^{*}\mid b_{\nu}\ne0\right\} .
\]

\end{defn}
Note that the ring $\mathcal{R}$ is not closed under the action of
the Weyl group. Let $\tilde{W}$ be a subgroup of the Weyl group.
We define subrings $\mathcal{R}_{\tilde{W}}$ and $\mathcal{R}'$
of $\mathcal{R}$ which are closed under the action of $\tilde{W}$.

Let $\mathcal{R}_{\tilde{W}}$ be the subalgebra of $\mathcal{R}$
defined by 
\[
\mathcal{R}_{\tilde{W}}:=\left\{ \sum_{\nu\in\hat{\mathfrak{h}}^{*}}b_{\nu}e^{\nu}\in\mathcal{R}\mid\sum_{\nu\in\hat{\mathfrak{h}}^{*}}b_{\nu}e^{w\nu}\in\mathcal{R}\mbox{ for all }w\in\tilde{W}\right\} 
\]
and $\mathcal{R}'$ the localization of $\mathcal{R}_{\tilde{W}}$
by 
\[
\mathcal{Y}:=\left\{ \prod_{\alpha\in X}\left(1+a_{\alpha}e^{-\alpha}\right)^{r\left(\alpha\right)}\mid a_{\alpha}\in\mathbb{Q},r\left(\alpha\right)\in\mathbb{Z}_{\ge0}\mbox{ and }X\subset\hat{\Delta},\left|X\backslash\hat{\Delta}^{+}\right|<\infty\right\} .
\]
The elements of $\mathcal{Y}$ are invertible in $\mathcal{R}$ using
geometric series (for example $\left(1-e^{-\alpha}\right)^{-1}$$=-e^{-\alpha}\left(1-e^{\alpha}\right)^{-1}$\\
$=-\sum_{i=1}^{\infty}e^{i\alpha}$) and $\mathcal{Y}$ is contained
in $\mathcal{R}_{\tilde{W}}$ (see \cite[1.4.2]{G1}). We extend the
action of $\tilde{W}$ from $\mathcal{R}_{\tilde{W}}$ to $\mathcal{R}'$
by $w\left(Y^{-1}Y'\right)=\left(wY\right)^{-1}\left(wY'\right)$
(see \cite[1.4.3]{G1}).
\begin{rem}
\label{rem:maximal element}Note that the maximal element in the support
of $\prod_{\alpha\in X}\left(1+a_{\alpha}e^{-\alpha}\right)^{r\left(\alpha\right)}\in\mathcal{Y}$
is \newline$-\sum_{\alpha\in X\backslash\hat{\Delta}^{+}:a_{\alpha}\ne0}r\left(\alpha\right)\nu$.
\end{rem}
For $Y\in\mathcal{R}_{\tilde{W}}$ such that each $\tilde{W}$-orbit
in $\hat{\mathfrak{h}}^{*}$ has a finite intersection with $\mbox{supp}Y$,
denote the sum
\[
\mathcal{F}_{\tilde{W}}\left(Y\right):=\sum_{w\in\tilde{W}}\mbox{sgn}\left(w\right)wY.
\]
We use the following lemmas from \cite[1.4.4]{G1}:
\begin{lem}
Suppose $Y\in\mathcal{R}_{\tilde{W}}$ and $\mathcal{F}_{\tilde{W}}\left(Y\right)\in\mathcal{R}$,
then $\mathcal{F}_{\tilde{W}}\left(Y\right)\in\mathcal{R}_{\tilde{W}}$
and is $\tilde{W}$-anti-invariant\label{lem:anti invariant}
\end{lem}
A set is called \emph{$\tilde{W}$-regular} if for every element $\lambda$
in the set, $\mbox{Stab}_{\tilde{W}}\lambda$ is trivial.
\begin{lem}
The support of a $\tilde{W}$-anti-invariant element in $\mathcal{R}_{\tilde{W}}$
is a union of regular $\tilde{W}$-orbits.\label{lem:support union regular orbits}\end{lem}
\begin{rem}
\label{rem:anti invariant R1}In order to use the above lemmas for
$\hat{R}$ and $\mathcal{F}_{T'}\left(e^{\hat{\rho}}R\right)\in\mathcal{R}'$
, we will multiply them by $\hat{R}_{1}\in\mathcal{Y}$. Note that
since $e^{\hat{\rho}}\hat{R}$ and $e^{\hat{\rho}'}\hat{R}_{\bar{0}}$
are $\hat{W}'$-skew invariant elements of $\mathcal{R}_{\hat{W}'}$,
$e^{\hat{\rho}'-\hat{\rho}}\hat{R}_{1}$ is $\hat{W}'$-invariant.
Similarly $e^{\hat{\rho}''-\hat{\rho}}\hat{R}_{1}$ is $\hat{W}''$-invariant.
\end{rem}

\section{main argument\label{sec:Domain U}}

In this section we use the Casimir operator and the denominator identity
for finite dimensional Lie superalgebras to show that the supports
of $e^{\hat{\rho}}\hat{R}$ and $\mathcal{F}_{T'}\left(e^{\hat{\rho}}R\right)$
belong to the following subset of $\hat{\mathfrak{h}}^{*}$:
\[
U:=\left\{ \mu\in\hat{\mathfrak{h}}^{*}\mid\left(\mu,\mu\right)=\left(\hat{\rho},\hat{\rho}\right)\right\} .
\]

Let us explain how this argument is used to prove Theorem \ref{thm:identity}.
First, using $\hat{W}'$-anti-invariance (namely lemmas \ref{lem:anti invariant}
and \ref{lem:support union regular orbits}), we get that it is sufficient
to check the denominator identity only on a small subset of the coefficients.
The fact that these are coefficients of elements in $U$, implies
that one has to check the identity on even fewer coefficients. When
$h^{\vee}\ne0$, the rest of the proof amounts to comparing the coefficient
of $e^{\hat{\rho}}$ on both sides of the identity. When $h^{\vee}=0$,
one should calculate the coefficients of the powers of $q$. This
is carried out in sections \ref{sec:h ne 0} and \ref{sec:h=00003D0},
respectively.

We shall use the following classical lemma (see for example \cite[10.4]{K}).
\begin{lem}
One has $\mbox{\emph{supp}}\left(e^{\hat{\rho}}\hat{R}\right)\subset U$. \end{lem}
\begin{proof}
Since $\hat{\mathfrak{g}}$ admits a Casimir element, the character
of the trivial $\hat{\mathfrak{g}}$-module is an integral linear
combination of the characters of Verma $\hat{\mathfrak{g}}$-modules
$M\left(\lambda\right)$, where $\lambda\in-\hat{Q}^{+}$, are such
that $\left(\lambda+\hat{\rho},\lambda+\hat{\rho}\right)=\left(\hat{\rho},\hat{\rho}\right)$
(see \cite[9.8]{K}). Since the character of $M\left(\lambda\right)$
is equal to $\hat{R}^{-1}e^{\lambda}$, we obtain 
\[
1=\sum_{\lambda\in-\hat{Q}^{+},\left(\lambda+\hat{\rho},\lambda+\hat{\rho}\right)=\left(\hat{\rho},\hat{\rho}\right)}a_{\lambda}\mbox{ch}M(\lambda)=\sum_{\lambda\in-\hat{Q}^{+},\left(\lambda+\hat{\rho},\lambda+\hat{\rho}\right)=\left(\hat{\rho},\hat{\rho}\right)}a_{\lambda}e^{\lambda}\hat{R}^{-1}
\]
where $a_{\lambda}\in\mathbb{Z}$. This can be rewritten as
\[
\hat{R}e^{\hat{\rho}}=\sum_{\lambda\in\hat{\rho}-\hat{Q}^{+},\left(\lambda,\lambda\right)=\left(\hat{\rho},\hat{\rho}\right)}a_{\lambda}e^{\lambda},
\]
that is $\mbox{supp}\left(e^{\hat{\rho}}\hat{R}\right)\subset U$. \end{proof}
\begin{lem}
\label{lem:The-support-of}The support of $\mathcal{F}_{T'}\left(e^{\hat{\rho}}R\right)$
is contained in $U$.\end{lem}
\begin{proof}
Let us show that for every $t\in T'$, the support of $t\left(e^{\hat{\rho}}R\right)$
is contained in $U.$ Recall that the denominator identity of finite
dimensional Lie superalgebras takes the form
\[
e^{\rho}R=\mathcal{F}_{W^{\#}}\left(\frac{e^{\rho}}{\prod_{\beta\in S}\left(1+e^{-\beta}\right)}\right)
\]
where $S$ is a maximal isotropic set of roots and the set of simple
roots of $\mathfrak{g}$ (which determines $\rho$) is assumed to
contain $S$ (see Section \ref{sub:W}). Since $\hat{\rho}-\rho$
is $W^{\#}$-invariant, we have 
\begin{eqnarray*}
t\left(e^{\hat{\rho}}R\right) & = & t\mathcal{F}_{W^{\#}}\left(\frac{e^{\hat{\rho}}}{\prod_{\beta\in S}\left(1+e^{-\beta}\right)}\right)\\
 & = & \sum_{w\in W^{\#}}\left(\mbox{sgn}w\right)\frac{e^{tw\hat{\rho}}}{\prod_{\beta\in S}\left(1+e^{-tw\beta}\right)}.
\end{eqnarray*}
For each $w\in W^{\#}$, the support of $e^{tw\hat{\rho}}\cdot\prod_{\beta\in S}\left(1+e^{-tw\beta}\right)^{-1}$
is contained in $tw\hat{\rho}+\mathbb{Z}\left\{ tw\beta\mid\beta\in S\right\} $
(see Remark \ref{rem:maximal element}). Since $\left(\hat{\rho},\beta\right)=0$
for all $\beta\in S$ and $\left(\cdot,\cdot\right)$ is invariant
under the action of the Weyl group and $T'$, the assertion follows.
\end{proof}

\section{Proof of the Denominator Identity, $h^{\vee}\ne0$\label{sec:h ne 0}}

We prove Theorem \ref{thm:identity} for the case $h^{\vee}\ne0$
in three steps. We follow the proof of \cite{G1}. The first step
of the proof is to show that the right hand side of the identity is
a well defined element of $\mathcal{R}'$. For this we use the properties
of roots systems described described in Proposition \ref{pro:properties of simple roots}).
The second step is to show that the support of the difference between
the two sides of the equation admits at most one maximal element which
is $\hat{\rho}$. The third step of the proof is to show that the
coefficient of $e^{\hat{\rho}}$ in both sides is $1$.

\subsection{Another form of the denominator identity}

We first rewrite $\mathcal{F}_{T'}\left(e^{\hat{\rho}}R\right)$ using
the denominator identity for finite dimensional Lie superalgebras
and use this form to prove Theorem \ref{thm:identity}.

Recall that $\pi$ contains a maximal set of isotropic roots $S$
(see Section \ref{sub:W}). In the first step of the proof we show
that $\mathcal{F}_{\hat{W}'}\left(e^{\hat{\rho}}\prod_{\beta\in S}\left(1+e^{-\beta}\right)^{-1}\right)$
is well defined. Recall that for $\hat{\mathfrak{g}}\ne A\left(2k-1,2l-1\right)$,
$k\ge l+1$, $\hat{W}'=T'\rtimes W'$. Then, using the denominator
identity for basic Lie superalgebras (\ref{eq:fin dim denom}), we
have that
\begin{eqnarray}
\mathcal{F}_{T'}\left(e^{\hat{\rho}}R\right) & = & \mathcal{F}_{T'}\left(e^{h^{\vee}\Lambda_{0}}e^{\rho}R\right)\nonumber \\
 & = & \mathcal{F}_{T'}\left(e^{h^{\vee}\Lambda_{0}}\cdot\mathcal{F}_{W'}\left(\frac{e^{\rho}}{\prod_{\beta\in S}\left(1+e^{-\beta}\right)}\right)\right)\nonumber \\
 & = & \mathcal{F}_{T'}\left(\mathcal{F}_{W'}\left(\frac{e^{h^{\vee}\Lambda_{0}+\rho}}{\prod_{\beta\in S}\left(1+e^{-\beta}\right)}\right)\right)\label{eq:other form}\\
 & = & \mathcal{F}_{\hat{W}'}\left(\frac{e^{\hat{\rho}}}{\prod_{\beta\in S}\left(1+e^{-\beta}\right)}\right).\nonumber 
\end{eqnarray}
For $\hat{\mathfrak{g}}=A\left(2k-1,2l-1\right)$, $k\ge l+1$, $\hat{W}'\not\supset T'=\mbox{span}_{\mathbb{Z}}\left\{ \varepsilon_{1},\ldots,\varepsilon_{k}\right\} $
(see Section \ref{sub:W}). We have that $\hat{W}'$ and $W'$ are
subgroups of index $2$ in $\hat{W}_{C_{k}}$ and $W_{C_{k}}$, respectively.
Recall that $\hat{W}_{C_{k}}=T'\rtimes W_{C_{k}}$ and the sign function
is extended from $\hat{W}'$ to $\hat{W}_{C_{k}}$ such that $\mbox{sgn}s_{\varepsilon_{k}}=1$.
Note that $S$ and $\hat{\rho}$ are $s_{\varepsilon_{k}}$-invariant
and hence 
\begin{eqnarray}
\mathcal{F}_{T'}\left(e^{\hat{\rho}}R\right) & = & \mathcal{F}_{T'}\left(e^{h^{\vee}\Lambda_{0}}\cdot\mathcal{F}_{W'}\left(\frac{e^{\rho}}{\prod_{\beta\in S}\left(1+e^{-\beta}\right)}\right)\right)\nonumber \\
 & = & \mathcal{F}_{T'}\left(e^{h^{\vee}\Lambda_{0}}\cdot\mathcal{F}_{W'}\left(\frac{1}{2}\left(\frac{e^{\rho}}{\prod_{\beta\in S}\left(1+e^{-\beta}\right)}+s_{2\varepsilon_{i}}\frac{e^{\rho}}{\prod_{\beta\in S}\left(1+e^{-\beta}\right)}\right)\right)\right)\nonumber \\
 & = & \frac{1}{2}\mathcal{F}_{T'}\left(e^{h^{\vee}\Lambda_{0}}\cdot\mathcal{F}_{W_{C_{k}}}\left(\frac{e^{\rho}}{\prod_{\beta\in S}\left(1+e^{-\beta}\right)}\right)\right)\nonumber \\
 & = & \frac{1}{2}\mathcal{F}_{\hat{W}_{C_{k}}}\left(\frac{e^{\hat{\rho}}}{\prod_{\beta\in S}\left(1+e^{-\beta}\right)}\right)\label{eq:other form case A(2k-1,2l-1)}\\
 & = & \frac{1}{2}\mathcal{F}_{\hat{W}'}\left(\frac{e^{\hat{\rho}}}{\prod_{\beta\in S}\left(1+e^{-\beta}\right)}+s_{2\varepsilon_{i}}\frac{e^{\hat{\rho}}}{\prod_{\beta\in S}\left(1+e^{-\beta}\right)}\right)\nonumber \\
 & = & \mathcal{F}_{\hat{W}'}\left(\frac{e^{\hat{\rho}}}{\prod_{\beta\in S}\left(1+e^{-\beta}\right)}\right)\nonumber 
\end{eqnarray}

We get that the denominator identity can be written as
\[
e^{\hat{\rho}}\hat{R}=\mathcal{F}_{\hat{W}'}\left(\frac{e^{\hat{\rho}}}{\prod_{\beta\in S}\left(1+e^{-\beta}\right)}\right).
\]
The fact that $\mathcal{F}_{\hat{W}'}\left(e^{\hat{\rho}}\prod_{\beta\in S}\left(1+e^{-\beta}\right)^{-1}\right)$
is well defined implies that so is $\mathcal{F}_{T'}\left(e^{\hat{\rho}}R\right)$.
However, the converse implication does not hold since it is not necessarily
possible to open the parenthesis in the third equality of (\ref{eq:other form})
and the fourth equality of (\ref{eq:other form case A(2k-1,2l-1)}).
In fact, neither $\mathcal{F}_{\hat{W}'}\left(e^{\hat{\rho}}\prod_{\beta\in S}\left(1+e^{-\beta}\right)^{-1}\right)$
nor $\mathcal{F}_{\hat{W}''}\left(e^{\hat{\rho}}\prod_{\beta\in S}\left(1+e^{-\beta}\right)^{-1}\right)$
is well defined for $\hat{\mathfrak{g}}=A\left(2k-1,2k-1\right)^{\left(2\right)}$.

\subsection{Choice of simple roots}

For the rest of this section we will assume that our choice of simple
roots satisfies the conditions of the following proposition. The proof
of existence of such choices is given in Appendix \ref{sub:Appendix description of simple roots}.
\begin{prop}
\label{pro:properties of simple roots}For every twisted affine Lie
superalgebra with $h^{\vee}\ne0$, there exists a choice of set of
simple roots $\hat{\pi}$ such that

\emph{(i)} for all $\alpha\in\hat{\pi}$, $\left(\alpha,\alpha\right)\ge0$;

\emph{(ii)} for $\hat{\mathfrak{g}}\ne A\left(2k,2k+1\right)^{\left(2\right)},\, A\left(2k,2k-1\right)^{\left(2\right)}$
and \textbf{$G\left(3\right)^{\left(2\right)}$}, one has $\left(\alpha_{0},\alpha_{0}\right)>0$. \end{prop}
\begin{rem}
\label{lem:positiveness}Suppose that the set of simple roots satisfies
(i), then $\left(\hat{\rho},\hat{Q}^{+}\right)\ge0$.
\end{rem}

Suppose that we fix the {}``finite part'' $\mathfrak{g}$ in $\hat{\mathfrak{g}}$,
then $e^{\hat{\rho}}\hat{R}$ and $e^{\hat{\rho}}R$ are independent
of the choice of simple roots. Thus, it suffices to prove Theorem
\ref{thm:identity} for a choice satisfying the conditions of Proposition
\ref{pro:properties of simple roots}. However, the choice of the
finite part is not unique. For example, the Lie algebra $B_{l}^{\left(1\right)}$
has two non-isomorphic maximal finite subalgebras, namely $B_{l}$
and $D_{l}$ each obtained by removing one vertex from the Dynkin
diagram of $B_{l}^{\left(1\right)}$. The situation is more complicated
for Lie superalgebras since the Dynkin diagram is not unique and one
can obtain different finite parts by removing a vertex in different
diagrams

The denominator identity for non-twisted Lie superalgebras (\cite{G1,GR})
also depends on the choice of the finite part. However, for twisted
algebras this problem is more significant since the choice of the
finite part is not canonical. Consider for example the first set of
simple roots described in Table 2. One can see that if we take the
finite part that corresponds to the Dynkin diagram of $\hat{\pi}\backslash\left\{ \varepsilon_{k}\right\} $,
the conditions of Proposition \ref{pro:properties of simple roots}
still hold and hence our proof for Theorem \ref{thm:identity} applies.
However, it is not clear whether one can also do so for other choices
of simple roots. It is interesting to look for a proof of the denominator
identity (twisted and not-twisted) that is independent of the choice
of simple roots and hence will apply to all finite parts.

\subsection{Step I }

Let us show that $\mathcal{F}_{\hat{W}'}\left(e^{\hat{\rho}}R\right)$
is well defined. By Section \ref{eq:other form}, $\mathcal{F}_{T'}\left(e^{\hat{\rho}}R\right)$
is well-defined as well. 
\begin{prop}
\label{pro:well defiinendness}The formal sum $\mathcal{F}_{\hat{W}'}\left(\frac{e^{\hat{\rho}}}{\prod_{\beta\in S}\left(1+e^{-\beta}\right)}\right)$
is in $\mathcal{R}'$ and its support lies in $\hat{\rho}-\hat{Q}^{+}$.\end{prop}
\begin{proof}
We extend the proof in \cite[2.4.1]{G1} to twisted affine Lie superalgebras.
 By \cite[5]{HS} and \cite[4.10]{S} the set of principal roots
of $\hat{\Delta}$ is equal to the set of simple roots of $\hat{\Delta}_{\bar{0}}^{+}$.
Hence  $\hat{W}'$ is a subgroup of the group $W_{+}$ introduced
in Lemma \ref{lem:WeylGroupAction} and $\Pi_{+}\subseteq\hat{\pi}'$.
By Remark \ref{rem:maximal element}, for every $w\in\hat{W}'$, the
maximal element of the support $w\left(e^{\hat{\rho}}\cdot\prod_{\beta\in S}\left(1+e^{-\beta}\right)^{-1}\right)$
is $w\hat{\rho}+\sum_{\beta\in S:w\beta\in\hat{\Delta}^{-}}w\beta$.
By Lemma \ref{lem:WeylGroupAction}.(i), $w\hat{\rho}\le\hat{\rho}$.
Hence 
\[
\mbox{supp}\left(w\frac{e^{\hat{\rho}}}{\prod_{\beta\in S}\left(1+e^{-\beta}\right)}\right)\subset w\hat{\rho}-\hat{Q}^{+}\subset\hat{\rho}-\hat{Q}^{+}.
\]
So it suffices to show that the set $G_{r}:=\left\{ w\in\hat{W}'\mid\mbox{ht}\left(\hat{\rho}-w\hat{\rho}+\sum_{\beta\in S:w\beta\in\hat{\Delta}^{-}}w\beta\right)\le r\right\} $
is finite for every $r$. This set is contained in the set $H_{r}:=\left\{ w\in\hat{W}'\mid\mbox{ht}\left(\hat{\rho}-w\hat{\rho}\right)\le r\right\} $.
We will show that $H_{r}$ is finite. By Lemma \ref{lem:WeylGroupAction}.(ii)
every element in $H_{r}$ is of the form $w_{0}s_{i_{1}}w_{1}\cdot\ldots\cdot s_{i_{j}}w_{j}$
where $w_{i}\in H_{0}$, $j\le r$ and $s_{i_{j}}$ are simple reflections
for all $i$. 

Let us show that $H_{0}$ is a finite subgroup of $\hat{W}'$. By
Lemma \ref{lem:WeylGroupAction}.(iii) $H_{0}$ is the subgroup generated
by $\left\{ s_{\alpha}\mid\alpha\in\Pi_{+}\mbox{ and }\left\langle \hat{\rho},\alpha^{\vee}\right\rangle =0\right\} $.
Let $\Sigma$ and $\Sigma_{0}$ be the Dynkin diagrams of $\hat{\Delta}'$
and $\left\{ \alpha\in\Pi_{+}\,\mid\,\left\langle \hat{\rho},\alpha^{\vee}\right\rangle =0\right\} $,
respectively. The inclusion $\Sigma_{0}\subset\Sigma$ is proper.
Indeed, since $h^{\vee}\ne0$, there exists a simple root $\alpha$
such that $\frac{1}{2}\left(\alpha,\alpha\right)=\left(\hat{\rho},\alpha\right)\ne0$.
Since $\left(\alpha,\alpha\right)\ne0$, either $\alpha$ or $2\alpha$
is a principal root. By Proposition \ref{pro:properties of simple roots}.(i),
$\left(\alpha,\alpha\right)>0$ . Thus, either $\alpha$ or $2\alpha$
belongs to $\hat{\Delta}'$, so $\left(\hat{\rho},\hat{\Delta}'\right)\ne0$
and hence $\Sigma_{0}\ne\Sigma$. Since $\Sigma$ is affine and indecomposable
and $\Sigma_{0}$ is a proper subdiagram of $\Sigma$, we get that
$\Sigma_{0}$ is of finite type and thus $H_{0}$ is finite as asserted. 
\end{proof}

\subsection{Step II}

The next step of the proof is to show that the support of the difference
between the two sides of (\ref{eq:denom}) admits at most one maximal
element which is $\hat{\rho}$. We do this by showing that $\hat{\rho}$
is the only element in $U$ which is a maximal element of a regular
$\hat{W}'$-orbit. We multiply $e^{\hat{\rho}}\hat{R}$ and $\mathcal{F}_{T'}\left(e^{\hat{\rho}}R\right)$
by $e^{\hat{\rho}'-\hat{\rho}}\hat{R}_{1}$ so that they will belong
to the algebra $\mathcal{R}_{\hat{W}'}$ where the action of $\hat{W}'$
can be applied to the support of a series (that is lemmas \ref{lem:anti invariant}
and \ref{lem:support union regular orbits} are applicable).
\begin{prop}
\label{pro:The-only-element maximal support}If $Y:=e^{\hat{\rho}}\hat{R}-\mathcal{F}_{T'}\left(e^{\hat{\rho}}R\right)$
is non zero, the only maximal element in the support of $Y$ is $\hat{\rho}$.
\end{prop}
To prove this proposition, we use the following lemma about affine
Lie algebras:
\begin{lem}
\emph{(\cite[3.1.1]{G1})}\label{lem:regularOrbit} Let $\mathfrak{a}$
be an affine Lie algebra with set of simple roots $\pi_{\mathfrak{a}}$,
and $W_{\mathfrak{a}}$ its Weyl group. Let $\rho_{\mathfrak{a}}$
be such that $\left(\rho_{\mathfrak{a}},\alpha\right)=\frac{1}{2}\left(\alpha,\alpha\right)$
for all $\alpha\in\pi_{\mathfrak{a}}$. Suppose $\lambda\in\sum_{\alpha\in\pi_{\mathfrak{a}}}\mathbb{Q}\alpha$
is such that $\lambda+\rho_{\mathfrak{a}}$ is a maximal element in
a regular $W_{\mathfrak{a}}$-orbit and $\left<\lambda,\alpha^{\vee}\right>\in\mathbb{Z}$
for all $\alpha\in\pi_{\mathfrak{a}}$. Then $\lambda\in\mathbb{Q}\delta$
where $\delta$ is the minimal imaginary root of $\mathfrak{a}$.
\end{lem}
\begin{proof}[Proof of Proposition \ref{pro:The-only-element maximal support}]  

Let $\mu$ be a maximal element in $\mbox{supp}Y$. Let us show
that $\mu=\hat{\rho}.$ By Section \ref{sec:Domain U}, $\left(\mu,\mu\right)=\left(\hat{\rho},\hat{\rho}\right)$.
The element $\mu+\hat{\rho}'-\hat{\rho}$ is a maximal element of
the support of $e^{\hat{\rho}'-\hat{\rho}}\hat{R}_{1}\cdot Y$. By
the $\hat{W}'$-invariance of $e^{\hat{\rho}'-\hat{\rho}}\hat{R}_{1}$
(see Remark \ref{rem:anti invariant R1}) and (\ref{eq:other form}),(\ref{eq:other form case A(2k-1,2l-1)}),
we have the following equality 
\begin{equation}
e^{\hat{\rho}'-\hat{\rho}}\hat{R}_{1}\cdot Y=e^{\hat{\rho}'}\hat{R}_{0}-\mathcal{F}_{\hat{W}'}\left(\frac{e^{\hat{\rho}'}\hat{R}_{1}}{\prod_{\beta\in S}\left(1+e^{-\beta}\right)}\right)\label{eq:summands}
\end{equation}
Note that the summands in the right hand side of (\ref{eq:summands})
are $\hat{W}'$-anti-invariant elements of $\mathcal{R}_{\hat{W}'}$
(using Lemma \ref{lem:anti invariant}). By Lemma \ref{lem:support union regular orbits},
the support of $e^{\hat{\rho}'-\hat{\rho}}\hat{R}_{1}\cdot Y$ is
a union of regular orbits. For an element in $\lambda=\mathbb{Q}\hat{\pi}$,
we write $\lambda=p'\left(\lambda\right)+p''\left(\lambda\right)$
where $p'\left(\lambda\right)\in\mathbb{Q}\hat{\pi}'$ and $p''\left(\lambda\right)\in\mathbb{Q}\pi''$.
One has that $p'\left(\mu+\hat{\rho}'-\hat{\rho}\right)=p'\left(\mu-\hat{\rho}\right)+\hat{\rho}'$
is a maximal element in its $\hat{W}'$-orbit. By Lemma \ref{lem:regularOrbit},
$p'\left(\mu-\hat{\rho}\right)=-s\delta$, $s\in\mathbb{Q}$. Recall
that 

\begin{eqnarray*}
\left(\hat{\rho},\hat{\rho}\right) & = & \left(\mu,\mu\right)\\
 & = & \left(\hat{\rho}+p'\left(\mu-\hat{\rho}\right)+p''\left(\mu-\hat{\rho}\right),\hat{\rho}+p'\left(\mu-\hat{\rho}\right)+p''\left(\mu-\hat{\rho}\right)\right)\\
 & = & \left(\hat{\rho}-s\delta+p''\left(\mu-\hat{\rho}\right),\hat{\rho}-s\delta+p''\left(\mu-\hat{\rho}\right)\right)
\end{eqnarray*}
which implies that 
\[
\left(p''\left(\mu-\hat{\rho}\right),p''\left(\mu-\hat{\rho}\right)\right)+2\left(\hat{\rho},-s\delta+p''\left(\mu-\hat{\rho}\right)\right)=0.
\]
One has $\hat{\rho}-s\delta+p''\left(\mu-\hat{\rho}\right)\in\mbox{supp}Y$
and by Lemma \ref{pro:well defiinendness}, $\mbox{supp}Y\subset\hat{\rho}-\hat{Q}^{+}$.
Hence $\left(\hat{\rho},-s\delta+p''\left(\mu-\hat{\rho}\right)\right)\le0$
by Remark \ref{lem:positiveness}. Since $(\cdot,\cdot)$ is negative
definite on $\Delta''$, $\left(p''\left(\mu-\hat{\rho}\right),p''\left(\mu-\hat{\rho}\right)\right)\le0$
and we get that $p''\left(\mu-\hat{\rho}\right)=0$ and hence $s=0$.
Thus, $\mu=\hat{\rho}$ and the assertion follows.   \end{proof}
\begin{rem}
When $h^{\vee}=0$, there are algebras for which there is no choice
a set of simple roots such that $\left(\hat{\rho},\hat{Q}^{+}\right)\ge0$
and hence we can not use the argument described in this proof. 
\end{rem}

\subsection{Step III}

Let us complete the proof of Theorem \ref{thm:identity} for $h^{\vee}\ne0$
by showing that the coefficients of $e^{\hat{\rho}}$ are equal on
both sides of the equation. That is, we show that $\hat{\rho}$ does
not belong to the support of $e^{\hat{\rho}}\hat{R}-\mathcal{F}_{T'}\left(e^{\hat{\rho}}R\right)$. 

Clearly the coefficient of $e^{\hat{\rho}}$ in $e^{\hat{\rho}}\hat{R}$
is $1$. On the other hand, we have:
\begin{prop}
\label{lem:rhoValueOne}The coefficient of $e^{\hat{\rho}}$ in $\mathcal{F}_{\hat{W}'}\left(\frac{e^{\hat{\rho}}}{\prod_{\beta\in S}\left(1+e^{-\beta}\right)}\right)$
is $1$.\end{prop}
\begin{proof}
Note that the coefficient of $e^{\hat{\rho}}$ in $\sum_{w\in\hat{W}'}\left(\mbox{sgn}w\right)w\left(\frac{e^{\hat{\rho}}}{\prod_{\beta\in S}\left(1+e^{-\beta}\right)}\right)$
is equal to $\sum_{w\in A}\left(\mbox{sgn}w\right)$ where $A:=\left\{ w\in\hat{W}'\,\mid\, w\hat{\rho}=\hat{\rho},w\beta\in\hat{\Delta}^{+}\mbox{ for all }\beta\in S\right\} $
(see Remark \ref{rem:maximal element}). We prove the proposition
by showing that $A=\left\{ 1\right\} $.\\

\textbf{Case 1: $\left(\alpha_{0},\alpha_{0}\right)>0$.} We show
that the stabilizer of $\hat{\rho}$ in $\hat{W}'$ is trivial. Similarly
to the argument of Proposition \ref{pro:well defiinendness}, Lemma
\ref{lem:WeylGroupAction}.(iii) yields that the stabilizer of $\hat{\rho}$
in $\hat{W}'$ is generated by reflections with respect to roots in
$\hat{\pi}'$. By Proposition \ref{pro:properties of simple roots}.(i),
$\alpha_{0}\in\hat{\pi}'$ but $s_{\alpha_{0}}$ is not in the stabilizer
since $\left(\alpha_{0},\alpha_{0}\right)=\frac{1}{2}\left(\alpha_{0},\hat{\rho}\right)\ne0$.
Hence the stabilizer is generated by reflections with respect to roots
in $\pi'$ so $\mbox{Stab}_{\hat{W}'}\hat{\rho}$ lies in $W'$ and
thus coincides with $\mbox{Stab}_{W'}\rho$. We get that the coefficient
of $e^{\hat{\rho}}$ in $\mathcal{F}_{\hat{W}'}\left(\frac{e^{\hat{\rho}}}{\prod_{\beta\in S}\left(1+e^{-\beta}\right)}\right)$
is the same as in $\mathcal{F}_{W'}\left(\frac{e^{\hat{\rho}}}{\prod_{\beta\in S}\left(1+e^{-\beta}\right)}\right)$.
Since $\hat{\rho}-\rho$ is $W'$-invariant, we get that it is equal
to the coefficient of $e^{\rho}$ in $\mathcal{F}_{W'}\left(\frac{e^{\rho}}{\prod_{\beta\in S}\left(1+e^{-\beta}\right)}\right)$.
 By the denominator identity for finite dimensional Lie superalgebras
it is equal to the coefficient of $e^{\rho}$ in $e^{\rho}R$ which
is clearly $1$.\\

\textbf{Case 2: $G\left(3\right)^{\left(2\right)}$.} In this case
$S=\left\{ \varepsilon_{3}-\varepsilon_{2}-\varepsilon_{1}\right\} $,
$\hat{\Delta}'=\hat{\Delta}_{\bar{0}}\backslash\left\{ 2s\delta\pm2\varepsilon_{3}\right\} _{s\in\mathbb{Z}}$
and $\hat{\rho}=3\Lambda_{0}-\varepsilon_{3}+\varepsilon_{1}+\varepsilon_{2}$.
By Lemma \ref{lem:WeylGroupAction}.(iii), the stabilizer of $\hat{\rho}$
is generated by reflections with respect to the principal roots. The
principal roots are $\left\{ 2\varepsilon_{1},2\varepsilon_{2},\delta-3\varepsilon_{2}-\varepsilon_{1}\right\} $
and so $\mbox{Stab}_{\hat{W}'}\hat{\rho}=\left\{ 1,s_{\delta-3\varepsilon_{2}-\varepsilon_{1}}\right\} $.
Since $s_{\delta-3\varepsilon_{2}-\varepsilon_{1}}\left(\varepsilon_{3}-\varepsilon_{2}-\varepsilon_{1}\right)\notin\hat{\Delta}^{+}$,
$A=\left\{ 1\right\} $.\\

\textbf{Case 3: $A\left(2k,2k+1\right)^{\left(2\right)}$.} In this
case $\hat{\rho}=\Lambda_{0}+\frac{1}{2}\left(\sum_{i=1}^{k}\left(\varepsilon_{i}-\delta_{i}\right)+\varepsilon_{k+1}\right)$
and we take $S=\left\{ \delta_{i}-\varepsilon_{i+1}\right\} _{i=1,\ldots,k}$.
Here $\hat{\Delta}'=\left\{ s\delta_{s\ne0},s\delta\pm\varepsilon_{g}\pm\varepsilon_{h},2s\delta\pm2\varepsilon_{h}\right\} $
and $T'=\left\{ t_{\mu}\,\mid\,\mu\in M\right\} $ where $M=\mbox{span}_{\mathbb{Z}}\left\{ \pm\varepsilon_{g}\pm\varepsilon_{h}\right\} $,
$s\in\mathbb{Z}$, $g\ne h$ and $1\le g,h\le k+1$.

Let $w\in A$. We show that $w=1$. Write $w=t_{\mu}y$ where $y\in W'$
and $\mu\in\mbox{span}_{\mathbb{Z}}\left\{ \pm\varepsilon_{g}\pm\varepsilon_{h}\right\} $.
Then $wS\subset\hat{\Delta}^{+}$ means that $w\left(\delta_{i}-\varepsilon_{i+1}\right)=\delta_{i}-y\varepsilon_{i+1}+\left(\mu,y\varepsilon_{i+1}\right)\delta\in\hat{\Delta}^{+}$.
Hence $\left(\mu,y\varepsilon_{i}\right)\ge0$ for all $i=2,\ldots,k+1$.
On the other hand, $\hat{\rho}=w\hat{\rho}$ means that
\begin{eqnarray*}
\hat{\rho} & = & y\hat{\rho}+h^{\vee}\mu-\left(\left(\hat{\rho},y^{-1}\mu\right)+\frac{\left(\mu,\mu\right)}{2}h^{\vee}\right)\delta\\
 & = & y\hat{\rho}+\mu-\frac{1}{2}\left(\left(\varepsilon_{1}+\ldots+\varepsilon_{k+1},y^{-1}\mu\right)+\left(y^{-1}\mu,y^{-1}\mu\right)\right)\delta.
\end{eqnarray*}
Write $y^{-1}\mu=\sum a_{i}\varepsilon_{i}$ and $\sum a_{i}=0\,(\mbox{mod}\,2)$.
Then 
\begin{eqnarray*}
0 & = & \left(\varepsilon_{1}+\ldots+\varepsilon_{k+1},y^{-1}\mu\right)+\left(y^{-1}\mu,y^{-1}\mu\right)\\
 & = & \sum a_{i}\left(a_{i}+1\right).
\end{eqnarray*}
Since $a_{i}\in\mathbb{Z}$, we get that $a_{i}\in\left\{ 0,-1\right\} $.
Since $a_{i}=\left(\mu,y\varepsilon_{i}\right)\ge0$ for $i=2,\ldots,k+1$,
we have $a_{2},\ldots,a_{k+1}=0$ and hence $a_{1}=0$. Hence $\mu=0$
and $w=y\in W'$. Since $w\hat{\rho}=\hat{\rho}$, we get that $w$
permutes $\varepsilon_{1}\ldots,\varepsilon_{k+1}$ (no sign change).
Note that $\delta_{i}-\varepsilon_{j}\in\Delta^{-}$ if $j\le i$
and so the only permutation $w$ such that $wS\subset\Delta^{+}$
is $1$. Thus, $A=\left\{ 1\right\} $. \\

\textbf{Case 4: }\textbf{\textcolor{black}{$A\left(2k,2k-1\right)^{\left(2\right)}$.
}}\textcolor{black}{In this case} $S=\left\{ \delta_{i}-\varepsilon_{i}\right\} _{i=1,\ldots,k}$
and \textcolor{black}{$\hat{\rho}=\Lambda_{0}+\frac{1}{2}\sum_{i=1}^{k}\left(\varepsilon_{i}-\delta_{i}\right)$.}
Here $\hat{\Delta}'=\left\{ s\delta_{s\ne0},s\delta\pm\varepsilon_{g}\pm\varepsilon_{h},s\delta\pm\varepsilon_{g},\left(2s+1\right)\delta\pm2\varepsilon_{g}\right\} $
where $s\in\mathbb{Z}$, $g\ne h$ and $1\le g,h\le k$. Let $w\in A$.
We show that $w=1$. Write $w=t_{\mu}y$ where $y\in W'$ and $\mu\in\mbox{span}_{\mathbb{Z}}\left\{ \varepsilon_{1},\dots,\varepsilon_{k}\right\} $.
Then $wS\subset\hat{\Delta}^{+}$ means that $w\left(\delta_{i}-\varepsilon_{i}\right)=\delta_{i}-y\varepsilon_{i}+\left(\mu,y\varepsilon_{i}\right)\delta\in\hat{\Delta}^{+}$.
Hence $\left(\mu,y\varepsilon_{i}\right)\ge0$ for all $i=1,\ldots,k$.
On the other hand, $\hat{\rho}=w\hat{\rho}$ means that
\begin{eqnarray*}
\hat{\rho} & = & w\hat{\rho}=y\hat{\rho}+h^{\vee}\mu-\left(\left(\hat{\rho},y^{-1}\mu\right)+\frac{\left(\mu,\mu\right)}{2}h^{\vee}\right)\delta\\
 & = & y\hat{\rho}+\mu-\frac{1}{2}\left(\left(\varepsilon_{1}+\ldots+\varepsilon_{k},y^{-1}\mu\right)+\left(\mu,\mu\right)\right)\delta.
\end{eqnarray*}
Since $\left(\varepsilon_{1}+\ldots+\varepsilon_{k},y^{-1}\mu\right)\ge0$
and $\left(\mu,\mu\right)\ge0$, we get that $\mu=0$ and so $w=y\in W'$.
Thus, $w\hat{\rho}=\hat{\rho}$ implies that $w$ permutes $\varepsilon_{1}\ldots,\varepsilon_{k}$
(no sign change). Note that $\delta_{i}-\varepsilon_{j}\in\Delta^{-}$
if $j<i$ and hence the only permutation $w$ such that $wS\subset\hat{\Delta}^{+}$
is $1$. Thus, $A=\left\{ 1\right\} $. 
\end{proof}

\section{Proof of the Denominator Identity, $h^{\vee}=0$\label{sec:h=00003D0}}

In this section we prove Theorem \ref{thm:identity} for the case
$h^{\vee}=0$\textcolor{black}{, in three steps. }The first step is
to show that the sum $\mathcal{F}_{T'}\left(Re^{\hat{\rho}}\right)$
is well defined and belongs to $\mathcal{R}$.  In the second step,
we show that $\hat{R}^{-1}e^{-\hat{\rho}}\cdot\mathcal{F}_{T'}\left(Re^{\hat{\rho}}\right)$
takes the form $f\left(q\right)$. In the third step we compute $f\left(q\right)$
using a proper evaluation. The case $h^{\vee}=0$ consists of the
algebras\textcolor{black}{{} $A\left(2k-1,2k-1\right)^{\left(2\right)}$,
$A\left(2k,2k\right)^{\left(4\right)}$ and $D\left(k+1,k\right)^{\left(2\right)}$}.
We first describe the even roots and the translation groups.

\subsection{Description of the root system and the Weyl group.}

We describe the set of even roots $\Delta_{\bar{0}}=\Delta'\sqcup\Delta''$
of $\mathfrak{g}$, and the translation groups $T'$ and $T''$ of
the Weyl group. We denote by $\mathbb{C}\left[t^{m},t^{-m}\right](\hat{\mathfrak{k}})$,
the affine Lie algebra which is isomorphic to $\hat{\mathfrak{k}}$
where $t\otimes g\in\hat{\mathfrak{k}}$, $g\in\mathfrak{k}$ is mapped
to $t^{m}\otimes g\in\mathbb{C}\left[t^{m},t^{-m}\right](\hat{\mathfrak{k}})$.

\subsubsection{$A\left(2k-1,2k-1\right)^{\left(2\right)}$. }

In this case $\mathfrak{g}=D\left(k,k\right)$. The set of even roots
of $\mathfrak{g}$ is $\Delta'\sqcup\Delta''$,
\[
\Delta'=\left\{ \delta_{i}\pm\delta_{j}\mid i\ne j\right\} ,\quad\Delta''=\left\{ \varepsilon_{i}\pm\varepsilon_{j}\mid i\ne j\right\} \cup\left\{ 2\varepsilon_{i}\right\} 
\]
where $i,j=1,\ldots,k$. The set of even roots of $\hat{\mathfrak{g}}$
is $\hat{\Delta}_{\bar{0}}=\hat{\Delta}'\cup\hat{\Delta}''$ where
$\hat{\Delta}'$ and $\hat{\Delta}''$ are the root system of $A_{2k-1}^{\left(2\right)}$.
The translation subgroups are $T'=\left\{ t_{\mu}\,\mid\,\mu\in M'\right\} $
where $M'=\mbox{span}_{\mathbb{Z}}\left\{ \delta_{1},\ldots,\delta_{k}\right\} $
and $T''=\left\{ t_{\mu}\,\mid\,\mu\in M''\right\} $ where $M''=\mbox{span}_{\mathbb{Z}}\left\{ \varepsilon_{i}\pm\varepsilon_{j}\right\} $.

Recall that $\hat{W}'\nsupseteq T'$ and we embed $W'$ and $\hat{W}'$
in $W_{C_{k}}=\left\langle W',s_{\delta_{k}}\right\rangle $ and $\hat{W}_{C_{k}}=\left\langle \hat{W}',s_{\delta_{k}}\right\rangle $
, respectively. One has $\hat{W}_{C_{k}}=T'\rtimes W_{C_{k}}$.

\subsubsection{$A\left(2k,2k\right)^{\left(4\right)}$.\label{sub:A(2k,2k)}}

In this case $\mathfrak{g}=B\left(k,k\right)$. The set of even roots
of $\mathfrak{g}$ is $\Delta'\sqcup\Delta''$, 
\[
\Delta'=\left\{ \delta_{i}\pm\delta_{j}\mid i\ne j\right\} \cup\left\{ \delta_{i}\right\} ,\quad\Delta''=\left\{ \varepsilon_{i}\pm\varepsilon_{j}\mid i\ne j\right\} \cup\left\{ 2\varepsilon_{i}\right\} 
\]
where $i,j=1,\ldots,k.$ The set of even roots of $\hat{\mathfrak{g}}$
is $\hat{\Delta}_{\bar{0}}=\hat{\Delta}'\cup\hat{\Delta}''$ where
$\hat{\Delta}'$ and $\hat{\Delta}''$ are the root systems of $\mathbb{C}\left[t^{4},t^{-4}\right]\left(A_{2k}^{\left(2\right)}\right)$
and $\mathbb{C}\left[t^{2},t^{-2}\right]\left(A_{2k}^{\left(2\right)}\right)$,
respectively. The translation subgroups are $T'=\left\{ t_{\mu}\,\mid\,\mu\in M'\right\} $
where $M'=\mbox{span}_{\mathbb{Z}}\left\{ 2\delta_{1},\ldots,2\delta_{k}\right\} $
and $T''=\left\{ t_{\mu}\,\mid\,\mu\in M''\right\} $ where $M''=\mbox{span}_{\mathbb{Z}}\left\{ 2\varepsilon_{1},\ldots,2\varepsilon_{k}\right\} $.
As we shall see, in this case it is possible to swap between $\hat{\Delta}'$
and $\hat{\Delta}''$ and the proof works.

\subsubsection{$D\left(k+1,k\right)^{\left(2\right)}$.}

In this case $\mathfrak{g}=B\left(k,k\right)$ as well, and the set
of even roots of $\mathfrak{g}$ is the same as in \ref{sub:A(2k,2k)}.
The set of even roots of $\hat{\mathfrak{g}}$ is $\hat{\Delta}_{\bar{0}}=\hat{\Delta}'\cup\hat{\Delta}''$
where $\hat{\Delta}'$ and $\hat{\Delta}''$ are the root systems
of $D_{k+1}^{\left(2\right)}$ and $\mathbb{C}\left[t^{2},t^{-2}\right]\left(C_{k}^{\left(1\right)}\right)$,
respectively. The translation subgroups are $T'=\left\{ t_{\mu}\,\mid\,\mu\in M'\right\} $
where $M'=\mbox{span}_{\mathbb{Z}}\left\{ 2\delta_{1},\ldots,2\delta_{k}\right\} $
and $T''=\left\{ t_{\mu}\,\mid\,\mu\in M''\right\} $ where $M''=\mbox{span}_{\mathbb{Z}}\left\{ 2\varepsilon_{1},\ldots,2\varepsilon_{k}\right\} $.
In this case, one can swap between $\hat{\Delta}'$ and $\hat{\Delta}''$.
As we shall see, in this case it is possible to swap between $\hat{\Delta}'$
and $\hat{\Delta}''$ and the proof works.

\subsection{Step I}

Let us show that $\mathcal{F}_{T'}\left(e^{\rho}R\right)$ is a well
defined element of $\mathcal{R}$. In the case $\hat{\mathfrak{g}}=A\left(2k-1,2k-1\right)^{\left(2\right)}$,
we use a method from \cite[2.1]{GR} and in the cases $A\left(2k,2k\right)^{\left(4\right)}$
and $D\left(k+1,k\right)^{\left(2\right)}$, we use the denominator
identity for $B\left(k,k\right)$.

Note that when $h^{\vee}=0$, $\hat{\rho}=\rho$.
\begin{lem}
For $\hat{\mathfrak{g}}=A\left(2k-1,2k-1\right)^{\left(2\right)}$,
$\mathcal{F}_{T'}\left(e^{\rho}R\right)$ is well defined and belongs
to the algebra $\mathcal{R}$.\end{lem}
\begin{proof}
Let us show that for every $w\in T'$, the maximal element of $\mbox{supp}\, w\left(e^{\rho}R\right)$
is less than $\rho+\sum_{\beta\in\Delta_{\bar{1}}^{+}}\beta$ and
that for every $\nu\le\rho+\sum_{\beta\in\Delta_{\bar{1}}^{+}}\beta$,
there are only finitely many $w\in T'$ such that the maximal element
of $\mbox{supp}\, w\left(e^{\rho}R\right)$ is larger than $\nu$.
The latter implies that $\mathcal{F}_{T'}(e^{\rho}R)$ is well defined,
whereas the former implies that 
\[
\text{supp}\bigl(\mathcal{F}_{T'}(e^{\rho}R)\bigr)\subset\rho+\sum_{\beta\in\Delta_{1}^{+}}\beta-\hat{Q}^{+},
\]
 that is $\mathcal{F}_{T'}(e^{\rho}R)$ belongs to $\mathcal{R}$.

One has 
\[
\max\mbox{supp}\, w\left(e^{\rho}R\right)=w\rho-\sum_{\alpha\in\Delta_{0}^{+}:w\alpha<0}w\alpha+\sum_{\alpha\in\Delta_{1}^{+}:w\alpha<0}w\alpha.
\]
Each $w\in T'$ is of the form $w=t_{\mu}$ where $\mu\in\sum_{i=1}^{n}\mathbb{Z}\delta_{i}$.
Note that for every $\beta\in\mathbb{Q}\pi$, $w\beta<0$ if and only
if $\left(\beta,\mu\right)>0$. We obtain that 
\begin{equation}
\max\mbox{supp}\, t_{\mu}\left(e^{\rho}R\right)=-v\left(\mu\right)+\left(v\left(\mu\right),\mu\right)\delta,\label{eq:maxsupp t_mu}
\end{equation}
where 
\[
v\left(\mu\right)=-\rho+\sum_{\alpha\in\Delta_{0}^{+}:t_{\mu}\left(\alpha\right)<0}\alpha-\sum_{\alpha\in\Delta_{1}^{+}:t_{\mu}\left(\alpha\right)<0}\alpha.
\]
We show that

(i) for every $\mu$ such that $t_{\mu}\in T',\quad\left(v\left(\mu\right),\mu\right)\le0$;

(ii) for every $N>0$, $\left\{ \mu\mid\left(v\left(\mu\right),\mu\right)\ge-N\right\} $
is a finite set.\\
By (\ref{eq:maxsupp t_mu}), we see that condition (ii) insures that
only finitely many maximal elements can apear above a certain weight and condition (i) means that for all $\mu$ one has $$\max\text{supp}(t_{\mu}(e^\rho R))\le -v(\mu)\leq \rho + \sum_{\beta\in\Delta_{1}^+} \beta.$$ 

Let us verify (i) and (ii). Recall that $\mu$ has the form $\mu=\sum_{i=1}^{k}n_{i}\delta_{i}$,
where $n_{i}\in\mathbb{Z}$. Write $v\left(\mu\right)=v'+v''$, where
$v'=\sum_{i=1}^{k}a_{i}\delta_{i}$ and $v''$ lies in the span of
the $\varepsilon_{i}$-s. Let us show that if $n_{i}>0$ then $a_{i}\le-\frac{1}{2}$
and if $n_{i}<0$ then $a_{i}\ge\frac{1}{2}$. We shall then have
that 
\[
\left(v\left(\mu\right),\mu\right)=\sum a_{i}n_{i}\le-\frac{1}{2}\sum_{n_{i}>0}n_{i}+\frac{1}{2}\sum_{n_{i}<0}n_{i}\le0,
\]
and hence the set $\left\{ \mu\mid\left(v\left(\mu\right),\mu\right)\ge-N\right\} $
is a subset of $\left\{ \sum_{i=1}^{k}n_{i}\delta_{i}\mid\frac{1}{2}\sum\left|n_{i}\right|<N\right\} $
which is finite.\textbf{\textcolor{black}{}}\\
One has $\rho=0$ and 
\begin{eqnarray*}
\Delta_{\bar{0}}^{+} & = & \left\{ \delta_{i}\pm\delta_{j}\mid1\le i<j\le k\right\} \cup\left\{ \varepsilon_{i}\pm\varepsilon_{j}\mid1\le i<j\le k\right\} \cup\left\{ 2\varepsilon_{i}\mid1\le i\le k\right\} \\
\Delta_{\bar{1}}^{+} & = & \left\{ \delta_{i}\pm\varepsilon_{j}\mid1\le i<j\le k\right\} \cup\left\{ \varepsilon_{i}\pm\delta_{j}\mid1\le i\le j\le k\right\} .
\end{eqnarray*}
Hence
\begin{eqnarray*}
\left\{ \alpha\in\Delta_{\bar{0}}^{+}\mid\left(\alpha,\mu\right)>0\right\}  & = & \left\{ \delta_{i}-\delta_{j}\mid i<j,n_{i}>n_{j}\right\} \cup\left\{ \delta_{i}+\delta_{j}\mid i<j,n_{i}+n_{j}>0\right\} \\
\left\{ \alpha\in\Delta_{\bar{1}}^{+}\mid\left(\alpha,\mu\right)>0\right\}  & = & \left\{ \varepsilon_{i}-\delta_{j}\mid i\le j,n_{j}<0\right\} \cup\left\{ \delta_{i}-\varepsilon_{j}\mid i<j,n_{i}>0\right\} \cup\left\{ \delta_{i}+\varepsilon_{j}\mid n_{i}>0\right\} ,
\end{eqnarray*}
where $1\le i,j\le k$. So for $n_{i}>0$, one has $a_{i}\le\left(2k-i-1\right)-\left(2k-i\right)=-1$
and for $n_{i}<0$, one has $a_{i}\ge-\left(i-1\right)+i=1$ as required.\end{proof}
\begin{rem}
Note that the above argument does not apply if one would take $T''$
instead of $T'$. For example, $1\in\mbox{supp}\, t_{n\varepsilon_{1}}\left(e^{\rho}R\right)$
for every $n\le0$ and so $\sum_{t\in T''}t\left(e^{\rho}R\right)$
is not well defined.\textcolor{black}{}
\end{rem}
A similar argument applies for the cases $A\left(2k,2k\right)^{\left(4\right)}$and
$D\left(k+1,k\right)^{\left(2\right)}$. However, we shall prove a
stronger statement:
\begin{lem}
\label{lem:well Definedness cases A D}For the cases $\mbox{\ensuremath{\hat{\mathfrak{g}}}}=A\left(2k,2k\right)^{\left(4\right)}$and
$D\left(k+1,k\right)^{\left(2\right)}$, we have 
\begin{equation}
\mathcal{F}_{T'}\left(e^{\rho}R\right)=\mathcal{F}_{\hat{W}'}\left(\frac{e^{\rho}}{\prod_{\beta\in S}\left(1+e^{-\beta}\right)}\right)\label{eq:A D beta sum}
\end{equation}
where $S$ is a maximal isotropic subset of $\pi$ and both sums are
well defined elements of $\mathcal{R}$. \end{lem}
\begin{proof}
Let us show that the right hand side of (\ref{eq:A D beta sum}) is
well defined. For every $y\in\hat{W}'$, we compute the maximal element
$u\left(y\right)$ of the support of $y\left(e^{\rho}\prod_{\beta\in S}\left(1+e^{-\beta}\right)^{-1}\right)$
and see that each maximal element appears finitely many times and
is less than or equal to $\max_{w\in W'}w\rho$.

Write $y=t_{\mu}w$ where $t_{\mu}\in T'$ and $w\in W'$. Then \textbf{\small 
\begin{eqnarray*}
u\left(t_{\mu}w\right) & = & t_{\mu}w\rho+\sum_{\beta\in S\,:\, t_{\mu}w\beta<0}t_{\mu}w\beta\\
 & \stackrel{\rho=-\frac{1}{2}\sum_{\beta\in S}\beta}{=} & -\frac{1}{2}\sum_{\beta\in S}\left(w\beta-\left(w\beta,\mu\right)\delta\right)+\sum_{\beta\in S\,:\,\left(\mu,w\beta\right)>0}\left(w\beta-\left(w\beta,\mu\right)\delta\right)+\sum_{\beta\in S\,:\,\left(\mu,w\beta\right)=0,\, w\beta<0}w\beta\\
 & = & w\rho+\sum_{\beta\in S\,:\, t_{\mu}w\beta<0}w\beta-\frac{1}{2}\sum_{\beta\in S}\left|\left(w\beta,\mu\right)\right|\delta.
\end{eqnarray*}
}Hence  $u\left(t_{\mu}w\right)=v'-m\delta$ for $v'\in Q$ only
when $m=\frac{1}{2}\sum_{\beta\in S}\left|\left(w\beta,\mu\right)\right|$,
which is possible only for finitely many $\mu$'s.

Let us prove equality (\ref{eq:A D beta sum}). In these cases, $\mathfrak{g}$
is isomorphic to $B\left(k,k\right)$ and the denominator identity
holds for $W'$ as well (see \cite[2.2]{G2}), that is
\[
e^{\rho}R=\mathcal{F}_{W'}\left(\frac{e^{\rho}}{\prod_{\beta\in S}\left(1+e^{-\beta}\right)}\right).
\]
Since $\hat{W}'=T'\rtimes W'$, the equality (\ref{eq:A D beta sum})
follows.
\end{proof}

\subsection{Step II}

In this step we show that $\hat{R}^{-1}e^{-\rho}\cdot\mathcal{F}_{T'}\left(Re^{\rho}\right)$
takes the form $f\left(q\right)$. As in the non-twisted case \cite[2.3.2]{GR},
this follows from a proposition stating that $\mbox{supp}\left(Y\right)\subset\hat{Q}^{\hat{W}}$.
For all twisted affine Lie superalgebras $\hat{Q}^{\hat{W}}=\mathbb{Z}\delta$,
completing this step of the proof. The proof of this proposition requires
the following two lemmas.
\begin{lem}
The term $e^{\hat{\rho}'-\rho}\hat{R}_{\bar{1}}\cdot\mathcal{F}_{T'}\left(e^{\rho}R\right)$
is a $\hat{W}'$-anti-invariant element of $\mathcal{R}_{\hat{W}'}$.\label{lem:hat W' anti invariant}\end{lem}
\begin{proof}
By Lemma \ref{lem:anti invariant}, it suffices to find $Y\in\mathcal{R}_{\hat{W}'}$
such that
\begin{equation}
e^{\hat{\rho}'-\rho}\hat{R}_{\bar{1}}\cdot\mathcal{F}_{T'}\left(e^{\rho}R\right)=\mathcal{F}_{\hat{W}'}\left(Y\right).\label{eq:F W Y}
\end{equation}
We will find $Y$ in the form $Y=e^{\hat{\rho}'-\rho}\hat{R}_{\bar{1}}\cdot Z$.
Since $e^{\hat{\rho}'-\rho}\hat{R}_{\bar{1}}$ is $\hat{W}'$-invariant
(see Remark \ref{rem:anti invariant R1}), the equality (\ref{eq:F W Y})
is equivalent to 
\begin{equation}
\mathcal{F}_{T'}\left(e^{\rho}R\right)=\mathcal{F}_{\hat{W}'}\left(Z\right).\label{eq:F W Z}
\end{equation}

For the cases $A\left(2k,2k\right)^{\left(4\right)}$ and $D\left(k+1,k\right)^{\left(2\right)}$,
we take $Z:=e^{\rho}\prod_{\beta\in S}\left(1+e^{-\beta}\right)^{-1}$.
Then, equality (\ref{eq:F W Z}) follows from Lemma \ref{lem:well Definedness cases A D}
and $Y=e^{\hat{\rho}'-\rho}\hat{R}_{\bar{1}}\cdot e^{\rho}\prod_{\beta\in S}\left(1+e^{-\beta}\right)^{-1}$
belongs to $\mathcal{R}_{\hat{W}'}$ by Section \ref{sub:formal power series}.

For the remaining $A\left(2k-1,2k-1\right)^{\left(2\right)}$ case,
we take $Z:=e^{\rho}R_{\bar{0}}''\cdot R_{\bar{1}}^{-1}$ where $R_{\bar{0}}''=\prod_{\alpha\in\Delta''^{+}}\left(1-e^{-\alpha}\right)$.
Then $Y=e^{\hat{\rho}'}\hat{R}_{\bar{1}}\cdot R_{\bar{0}}''\cdot R_{\bar{1}}^{-1}$
is again in $\mathcal{R}_{\hat{W}'}$ by Section \ref{sub:formal power series}.
Let us prove equality (\ref{eq:F W Z}). Note that $e^{\rho}R_{\bar{0}}''$
is $\hat{W}'$-invariant and $\rho=0$. Dividing both sides of equality
(\ref{eq:F W Z}) by $e^{\rho}R_{\bar{0}}''$ we obtain the equivalent
equality
\begin{equation}
\mathcal{F}_{T'}\left(R_{\bar{0}}'\cdot R_{\bar{1}}^{-1}\right)=\mathcal{F}_{\hat{W}'}\left(R_{\bar{1}}^{-1}\right).\label{eq:F W R1}
\end{equation}
Note that the left hand side of (\ref{eq:F W R1}) belongs to $\mathcal{R}'$
because $\mathcal{F}_{T'}\left(e^{\rho}R\right)\in\mathcal{R}$. 

Let us show that $\mathcal{F}_{\hat{W}'}\left(R_{\bar{1}}^{-1}\right)$
is well defined and belongs to $\mathcal{R}$. For every $y\in\hat{W}'$,
one has 
\[
\mbox{max}\mbox{supp}\left(yR_{\bar{1}}^{-1}\right)=\sum_{\beta\in\Delta_{\bar{1}}^{+}\,:\, y\beta<0}y\beta.
\]
Let us show that for every $\nu\in-\hat{Q}^{+}$, there are only finitely
many $y\in\hat{W}'$ such that $\mbox{max}\mbox{supp}\left(yR_{\bar{1}}^{-1}\right)=\nu$.
Write $y=t_{\mu}w$ where $t_{\mu}\in T'$ and $w\in W'$. Note that
$y\beta<0$ if $\left(\mu,w\beta\right)>0$ and that $\left(\mu,w\beta\right)=0$
implies that $y\beta\in Q$. Since $\mu\in\mbox{span}\left\{ \delta_{i}\right\} _{i=1}^{k}$,
for every $i=1,\ldots k$ one has either $\left(\mu,w\left(\varepsilon_{i}-\delta_{i}\right)\right)>0$
or $\left(\mu,w\left(\varepsilon_{i}+\delta_{i}\right)\right)>0$.
Write $\nu=-m\delta+\nu'$ where $\nu'\in Q$ and $\mu=\sum_{i=1}^{k}a_{i}\delta_{i}$.
We get that $\sum_{i=1}^{k}\left|a_{i}\right|\le m$ which is possible
only for finitely many $\mu$'s. Thus, the sum $\mathcal{F}_{\hat{W}'}\left(R_{\bar{1}}^{-1}\right)$
is well defined.

We are left to verify equality (\ref{eq:F W R1}). Recall that in
this case $\hat{W}'\not\nsupseteq T'$. The groups $\hat{W}'$ and
$W'$ are extended to $\hat{W}_{C_{k}}=\left\langle \hat{W}',s_{\delta_{k}}\right\rangle $
and $W_{C_{k}}=\left\langle W',s_{\delta_{k}}\right\rangle $, respectively,
and $\hat{W}_{C_{k}}=T'\rtimes W_{C_{k}}$. Recall that the sign function
is extended from $\hat{W}'$ to $\hat{W}_{C_{k}}$ such that $\mbox{sgn}s_{\delta_{k}}=1$
(see \ref{sub:W}). One has 
\[
\mathcal{F}_{\hat{W}_{C_{k}}}\left(R_{\bar{1}}^{-1}\right)=\mathcal{F}_{\hat{W}_{C_{k}}}\left(R_{\bar{1}}^{-1}\right)+\mathcal{F}_{\hat{W}'}s_{\delta_{k}}\left(R_{\bar{1}}^{-1}\right)=2\mathcal{F}_{\hat{W}'}\left(R_{\bar{1}}^{-1}\right).
\]
Since $\rho=0$, $e^{\rho'}R_{1}$ is $W_{C_{k}}$-invariant (see
Remark \ref{rem:anti invariant R1}). We obtain
\begin{eqnarray*}
\mathcal{F}_{\hat{W}_{C_{k}}}\left(R_{\bar{1}}^{-1}\right) & = & \mathcal{F}_{T'}\left(\mathcal{F}_{W_{C_{k}}}\left(R_{\bar{1}}^{-1}\right)\right)\\
 & = & \mathcal{F}_{T'}\left(e^{-\rho'}R_{\bar{1}}^{-1}\cdot\mathcal{F}_{W_{C_{k}}}\left(e^{\rho'}\right)\right)\\
 & = & \mathcal{F}_{T'}\left(e^{-\rho'}R_{\bar{1}}^{-1}\cdot\mathcal{F}_{W'}\left(e^{\rho'}+s_{\delta_{k}}e^{\rho'}\right)\right)\\
 & = & 2\mathcal{F}_{T'}\left(R_{\bar{0}}'\cdot R_{\bar{1}}^{-1}\right)
\end{eqnarray*}
as required.B\end{proof}
\begin{lem}
For $\hat{\mathfrak{g}}=A\left(2k,2k\right)^{\left(4\right)}$ or
$D\left(k+1,k\right)^{\left(2\right)}$, the term $e^{\hat{\rho}''-\rho}\hat{R}_{\bar{1}}\cdot\mathcal{F}_{T'}\left(e^{\rho}R\right)$
is a $\hat{W}''$-anti-invariant element of $\mathcal{R}_{\hat{W}''}$.
\label{lem:hat W'' ant invariant case A D}\end{lem}
\begin{proof}
Note that $e^{\hat{\rho}'-\rho}\hat{R}_{\bar{1}}$ is $\hat{W}'$-invariant
and by Section \ref{sub:formal power series}, 
\[
e^{\hat{\rho}''-\rho}\hat{R}_{\bar{1}}\cdot\mathcal{F}_{T'}\left(e^{\rho}R\right)=e^{\hat{\rho}''-\hat{\rho}'}\mathcal{F}_{T'}\left(e^{\hat{\rho}'}\hat{R}_{\bar{1}}\cdot R\right).
\]
Let us show $\hat{W}''$-anti-invariance. The term $e^{\hat{\rho}''-\rho}\hat{R}_{\bar{1}}$
is $\hat{W}''$-invariant (Remark \ref{rem:anti invariant R1}) and
so it suffices to show that $\mathcal{F}_{T'}\left(e^{\rho}R\right)$
is $\hat{W}''$-anti-invariant. Note that $\hat{W}'$ and $\hat{W}''$
commute. The anti-invariance with respect to $W''$ follows from the
one of $e^{\rho}R$. It remains to show invariance with respect to
$T''$. By Lemma \ref{lem:well Definedness cases A D}, $\mathcal{F}_{T'}\left(e^{\rho}R\right)=\mathcal{F}_{\hat{W}'}\left(e^{\rho}\prod_{\beta\in S}\left(1+e^{-\beta}\right)^{-1}\right)$.
Since $\rho=\frac{1}{2}\sum_{i=1}^{k}\left(\delta_{i}-\varepsilon_{i}\right)$
and $S=\left\{ \varepsilon_{i}-\delta_{i}\right\} _{i=1}^{k}$, 
\[
t_{\sum2n_{i}\varepsilon_{i}}\left(e^{\rho}\prod_{\beta\in S}\left(1+e^{-\beta}\right)^{-1}\right)=t_{\sum2n_{i}\delta_{i}}\left(e^{\rho}\prod_{\beta\in S}\left(1+e^{-\beta}\right)^{-1}\right)
\]
 and the assertion follows.\end{proof}
\begin{prop}
The support of $e^{-\rho}\hat{R}^{-1}\cdot\mathcal{F}_{T'}\left(e^{\rho}R\right)$
is contained in $\mathbb{Z}\delta$.\end{prop}
\begin{proof}
We generalize the argument of \cite[2.3.2]{GR}. Let $Y:=e^{-\rho}\hat{R}^{-1}\cdot\mathcal{F}_{T'}\left(e^{\rho}R\right)$.
Note that $e^{\hat{\rho}'-\rho}\hat{R}_{\bar{1}}\cdot\mathcal{F}_{T'}\left(e^{\rho}R\right)=e^{\hat{\rho}'}\hat{R}_{\bar{0}}\cdot Y$
which is a $\hat{W}'$-anti-invariant element of $\mathcal{R}_{\hat{W}'}$
by Lemma \ref{lem:hat W' anti invariant}. Write $Y=Y_{1}+Y_{2}$
where $\mbox{supp}\left(Y_{1}\right)\subset\mathbb{Z}\delta$ and
$\mbox{supp}\left(Y_{2}\right)\cap\mathbb{Z}\delta=\emptyset$. Since
$e^{\hat{\rho}'}\hat{R}_{\bar{0}}$ and $Y_{1}$ are $\hat{W}$-anti-invariant
and invariant, respectively, the term 
\[
e^{\hat{\rho}'}\hat{R}_{\bar{0}}\cdot Y_{2}=e^{\hat{\rho}'}\hat{R}_{\bar{0}}\cdot Y-e^{\hat{\rho}'}\hat{R}_{\bar{0}}\cdot Y_{1}
\]
 is also a $\hat{W}'$-anti-invariant element of $\mathcal{R}_{\hat{W}'}$.
We would like to show that this term is equal to zero. Let us show
that every maximal element $\mu\in\mbox{supp}\left(Y_{2}\right)$
belongs to $\mathbb{Z}\delta$, in contradiction to the definition
of $Y_{2}$. The element $\mu+\hat{\rho}'$ is maximal in $\mbox{supp}\left(\hat{R}_{\bar{0}}e^{\hat{\rho}'}\cdot Y_{2}\right)$.
By Lemma \ref{lem:support union regular orbits}, $\mbox{supp}\left(\hat{R}_{\bar{0}}e^{\hat{\rho}'}\cdot Y_{2}\right)$
is a union of regular orbits. Since $\mu+\hat{\rho}'$ is maximal
in a regular $\hat{W}'$-orbit and $\frac{2\left(\hat{\rho}',\alpha\right)}{\left(\alpha,\alpha\right)}=1$
for all $\alpha\in\hat{\pi}'$, one has $\left(\mu,\alpha\right)\ge0$.
On the other hand, $\left(\mu,\delta\right)=0$ and $\delta\in\mathbb{Z}_{>0}\hat{\pi}'$
and hence $\left(\mu,\hat{\pi}'\right)=0$. Since $\mathbb{Z}\hat{\Delta}''$
is the orthogonal set to $\hat{\pi}'$ in $\hat{Q}$, we get that
$\mu\in\mathbb{Z}\hat{\Delta}''$. 

For the cases $A\left(2k,2k\right)^{\left(4\right)}$ and $D\left(k+1,k\right)^{\left(2\right)}$,
we can interchange $\hat{W}'$ by $\hat{W}''$ and apply the same
argument (using Lemma \ref{lem:hat W'' ant invariant case A D} instead
of Lemma \ref{lem:hat W' anti invariant}). Thus, $\mu\in\hat{\Delta}'\cap\hat{\Delta}''$
and hence $\mu\in\mathbb{Z}\delta.$ 

For the case $A\left(2k-1,2k-1\right)^{\left(2\right)}$, let us show
that the support of $\mu+\rho$ belongs to $U$. Note that $\hat{R}e^{\rho}Y_{2}=\mathcal{F}_{T'}\left(Re^{\rho}\right)-\hat{R}e^{\rho}Y_{1}$
and by Section \ref{sec:Domain U}, 
\[
\mbox{supp}\left(\mathcal{F}_{T'}\left(Re^{\rho}\right)\right),\mbox{supp}\left(\hat{R}e^{\rho}\right)\subset U.
\]
Since $\left(\delta,\rho\right)=\left(\delta,\hat{Q}\right)=\left(\delta,\delta\right)=0$,
$U+\mathbb{Z}\delta\subset U$. Hence, $\mbox{supp}\left(\hat{R}e^{\rho}Y_{1}\right)\subset U$
implying that $\mbox{supp}\left(\hat{R}e^{\rho}Y_{2}\right)\subset U$
and so $\mu+\rho\in U$. 

Thus, $\left(\mu+\rho,\mu+\rho\right)=\left(\rho,\rho\right)=0$.
This is equivalent to $\left(\mu,\mu\right)=0$ because $\rho=0$.
Since the bilinear form is negative definite on $\Delta''$, we get
that $\mu\in\mathbb{Z}\delta$ and the proposition follows. \end{proof}
\begin{rem}
When $h^{\vee}\ne0$ we can not use this argument since $\left(\delta,\hat{\rho}\right)\ne0$
and $\mbox{supp}\left(\hat{R}e^{\hat{\rho}}Y_{1}\right)\not\subset U$
.
\end{rem}

\subsection{Step III}

In this section we compute $\mathcal{F}_{T'}\left(e^{\rho}R\right)\cdot e^{-\rho}\hat{R}^{-1}$,
knowing that it depends only on $q$, we describe an evaluation of
the variables $e^{-\alpha}$, $\alpha\in\pi$, in which $\mathcal{F}_{T'}\left(e^{\rho}R\right)\cdot e^{-\rho}R^{-1}$
is equal to $1$ and $R\cdot\hat{R}^{-1}$ can be easily computed.
We use the property of the algebras \textcolor{black}{$A\left(2k-1,2k-1\right)^{\left(2\right)}$,
$A\left(2k,2k\right)^{\left(4\right)}$ and $D\left(k+1,k\right)^{\left(2\right)}$
that }$\left|\Delta_{\bar{1}}^{+}\right|-\left|\Delta_{\bar{0}}^{+}\right|$
is equal to the defect which is $k$.

Let $x\in\mathbb{C}\backslash\left\{ 0\right\} $ and evaluate $e^{-\alpha}$
by $\left(-1\right)^{p\left(\alpha\right)}\cdot x$ for every $\alpha\in\pi$,
where $p\left(\alpha\right)\in\left\{ \bar{0},\bar{1}\right\} $ denotes
the parity of $\alpha$. It implies that $e^{-\gamma}$ is evaluated
by $\left(-1\right)^{p\left(\gamma\right)}x^{\mbox{ht}\left(\gamma\right)}$
for every $\gamma\in\Delta$. 
\begin{lem}
 Let $\hat{\mathfrak{g}}$ be one of the algebras \textcolor{black}{$A\left(2k-1,2k-1\right)^{\left(2\right)}$,
$A\left(2k,2k\right)^{\left(4\right)}$ and $D\left(k+1,k\right)^{\left(2\right)}$
}. Then for every $t\in T'$, $t\ne id$, $\left.\frac{\left(t\left(e^{\rho}R\right)\right)\left(x\right)}{\left(e^{\rho}R\right)\left(x\right)}\right|_{x=1}$
is equal to $0$.\end{lem}
\begin{proof}
One has that $e^{\rho}\left(x\right)=x^{n}$ for some $n\in\frac{1}{2}\mathbb{Z}$
and 
\[
R\left(x\right)=\frac{\prod_{\gamma\in\Delta_{\bar{0}}^{+}}\left(1-x^{\mbox{ht}\gamma}\right)}{\prod_{\gamma\in\Delta_{\bar{1}}^{+}}\left(1-x^{\mbox{ht}\gamma}\right)},
\]
and hence at $x=1$, the function $e^{\rho}R\left(x\right)$ has a
pole of order $\left|\Delta_{\bar{1}}^{+}\right|-\left|\Delta_{\bar{0}}^{+}\right|=k$. 

Let us show that if $t_{\mu}\ne id$, then $\left(t_{\mu}\left(e^{\rho}R\right)\right)\left(x\right)$
has a pole at $x=1$ of order which is strictly less than $k$. By
the denominator identity of finite dimensional Lie superalgebras (see
(\ref{eq:fin dim denom}))
\[
t_{\mu}\left(e^{\rho}R\right)=\sum_{w\in W^{''}}\left(-1\right)^{l\left(w\right)}\frac{e^{t_{\mu}w\rho}}{\prod_{i=1}^{k}\left(1+e^{-t_{\mu}w\beta_{i}}\right)}
\]
One has $t_{\mu}w\beta_{i}=w\beta_{i}+n_{i}\delta$ where $n_{i}=\left(\mu,w\beta_{i}\right)\in\mathbb{Z}$.
Hence the evaluation of $\left(1+e^{-tw\beta_{i}}\right)^{-1}$ is
equal to $\left(1-x^{m}q^{n_{i}}\right)^{-1}$ for some $m\in\mathbb{Z}_{\ne0}$.
Hence it has a pole at $x=1$ if and only if $n_{i}=0$. Thus, the
evaluation of $t_{\mu}\left(e^{\rho}R\right)$ has a pole of order
less or equal to the number of $n_{i}$'s which are zero. This number
is equal to $k$ if and only if $\left(\mu,\delta_{i}\right)=0$ for
$i=1,\ldots,k$. Since $\mu\in\mbox{span}\left\{ \delta_{1},\ldots,\delta_{k}\right\} $,
we get that $\mu=0$, that is $t=id$ and the assertion follows.\end{proof}
\begin{rem}
The above argument is based on the fact that the rank of $T'$ is
equal to the defect. In particular, it can not be used for the case
$h^{\vee}\ne0$.
\end{rem}
Let us compute $\left.\frac{\hat{R}}{R}\left(x\right)\right|_{x=1}$.\textbf{\textcolor{black}{}}\\
\textbf{\textcolor{black}{Case $A\left(2k-1,2k-1\right)^{\left(2\right)}$:}}
One has 
\[
\frac{\hat{R}}{R}=\prod_{n=1}^{\infty}\frac{\prod_{\alpha\in\Delta_{\bar{0}}^{(n\textrm{mod}2)}}\left(1-q^{n}e^{\alpha}\right)\left(1-q^{2n}\right)^{\dim\mathfrak{h}}\left(1-q^{2n+1}\right)^{\dim\hat{\mathfrak{g}}_{\delta}}}{\prod_{\alpha\in\Delta_{\bar{1}}^{(n\textrm{mod}2)}}\left(1+q^{n}e^{\alpha}\right)}.
\]
Here $\left|\Delta_{\bar{0}}^{(0)}\right|=\left|\Delta_{\bar{0}}^{(1)}\right|=4k^{2}-2k$
, $\left|\Delta_{\bar{1}}^{(0)}\right|=\left|\Delta_{\bar{1}}^{(1)}\right|=4k^{2}$,
$\dim\mathfrak{h}=2k$ and $\dim\hat{\mathfrak{g}}_{\delta}=2k-2$.
So 
\[
\left.\frac{\hat{R}}{R}\left(x\right)\right|_{x=1}=\prod_{n=1}^{\infty}\left(1-q^{2n+1}\right)^{-2}.
\]
\textcolor{black}{}\\
\textbf{\textcolor{black}{Case $A\left(2k,2k\right)^{\left(4\right)}$:}}
One has 
\[
\frac{\hat{R}}{R}=\prod_{n=1}^{\infty}\frac{\prod_{\alpha\in\Delta_{\bar{0}}^{(n\textrm{mod}4)}}\left(1-q^{n}e^{\alpha}\right)\left(1-q^{4n}\right)^{\dim\mathfrak{h}}\left(1-q^{4n+2}\right)^{\dim\hat{\mathfrak{g}}_{2\delta}}}{\prod_{\alpha\in\Delta_{\bar{1}}^{(n\textrm{mod}4)}}\left(1+q^{n}e^{\alpha}\right)\left(1+q^{4n+1}\right)^{\dim\hat{\mathfrak{g}}_{\delta}}\left(1+q^{4n+3}\right)^{\dim\hat{\mathfrak{g}}_{3\delta}}}.
\]
Here $\left|\Delta_{\bar{0}}^{(0)}\right|=\left|\Delta_{\bar{0}}^{(2)}\right|=4k^{2}$,
$\left|\Delta_{\bar{0}}^{(1)}\right|=\left|\Delta_{\bar{0}}^{(3)}\right|=2k$,
$\left|\Delta_{\bar{1}}^{(0)}\right|=\left|\Delta_{\bar{1}}^{(2)}\right|=4k^{2}+2k$,
$\left|\Delta_{\bar{1}}^{(1)}\right|=\left|\Delta_{\bar{1}}^{(3)}\right|=2k$,
$\dim\mathfrak{h}=\dim\hat{\mathfrak{g}}_{2\delta}=2k$ and $\dim\hat{\mathfrak{g}}_{\delta}=\dim\hat{\mathfrak{g}}_{3\delta}=1$.
So 
\[
\left.\frac{\hat{R}}{R}\left(x\right)\right|_{x=1}=\prod_{n=1}^{\infty}\left(1+q^{2n+1}\right)^{-1}.
\]
\textcolor{black}{}\\
\textbf{\textcolor{black}{Case $D\left(k+1,k\right)^{\left(2\right)}$:}}
One has 
\[
\frac{\hat{R}}{R}=\prod_{n=1}^{\infty}\frac{\prod_{\alpha\in\Delta_{\bar{0}}^{(n\textrm{mod}2)}}\left(1-q^{n}e^{\alpha}\right)\left(1-q^{2n}\right)^{\dim\mathfrak{h}}\left(1-q^{2n+1}\right)^{\dim\hat{\mathfrak{g}}_{\delta}}}{\prod_{\alpha\in\Delta_{\bar{1}}^{(n\textrm{mod}2)}}\left(1+q^{n}e^{\alpha}\right)}.
\]
Here $\left|\Delta_{\bar{0}}^{(0)}\right|=4k^{2}$ , $\left|\Delta_{\bar{0}}^{(1)}\right|=\left|\Delta_{\bar{1}}^{(1)}\right|=2k$,
$\left|\Delta_{\bar{1}}^{(0)}\right|=4k^{2}+2k$, $\dim\mathfrak{h}=2k$
and $\dim\hat{\mathfrak{g}}_{\delta}=1$. So 
\[
\left.\frac{\hat{R}}{R}\left(x\right)\right|_{x=1}=\prod_{n=1}^{\infty}\left(1-q^{2n+1}\right).
\]

\begin{rem}
For the computation of $\dim\mathfrak{g}_{\delta}$ see \cite[7.5.13]{vdL2}
or note that for $\tilde{\mathfrak{g}}=A\left(2k-1,2k-1\right)$,
$D\left(k+1,k\right)$, 
\begin{eqnarray*}
\dim\tilde{\mathfrak{g}}_{\bar{0}} & = & \left|\Delta_{\bar{0}}^{(0)}\right|+\left|\Delta_{\bar{0}}^{(1)}\right|+\dim\mathfrak{h}+\dim\hat{\mathfrak{g}}_{\delta}\\
\dim\tilde{\mathfrak{g}}_{\bar{1}} & = & \left|\Delta_{\bar{1}}^{(0)}\right|+\left|\Delta_{\bar{1}}^{(1)}\right|
\end{eqnarray*}
and for $\tilde{\mathfrak{g}}=A(2k,2k)$,
\begin{eqnarray*}
\dim\tilde{\mathfrak{g}}_{\bar{0}} & = & \left|\Delta_{\bar{0}}^{(0)}\right|+\left|\Delta_{\bar{0}}^{(1)}\right|+\left|\Delta_{\bar{0}}^{(2)}\right|+\left|\Delta_{\bar{0}}^{(3)}\right|+\dim\mathfrak{h}+\dim\hat{\mathfrak{g}}_{2\delta}\\
\dim\tilde{\mathfrak{g}}_{\bar{1}} & = & \left|\Delta_{\bar{1}}^{(0)}\right|+\left|\Delta_{\bar{1}}^{(1)}\right|+\left|\Delta_{\bar{1}}^{(2)}\right|+\left|\Delta_{\bar{1}}^{(3)}\right|+\dim\hat{\mathfrak{g}}_{\delta}+\dim\hat{\mathfrak{g}}_{3\delta}.
\end{eqnarray*}

\end{rem}

\section{Appendix}

We recall the construction of the twisted affine Lie superalgebras
and the description of their root systems. We list choices of simple
roots which we use to prove the denominator identity.

\subsection{A construction of the twisted affine Lie superalgebras.\label{sub:construction}}

We describe the automorphisms which are used in \cite{vdL2} to construct
the twisted affine Lie superalgebras. For every algebra we show how
the automorphism acts on the Chevalley generators $e_{\alpha_{i}},f_{\alpha_{i}}$
with respect to a standard choice of simple roots $\pi=\left\{ \alpha_{1},\dots,\alpha_{n}\right\} $
(a choice that contains at most one isotropic root). For $\alpha\in\tilde{\Delta}^{+}$,
let $e_{\alpha}:=\left[e_{\alpha_{i_{1}}},\left[e_{\alpha_{i_{2}}},\ldots\left[e_{\alpha_{i_{m-1}}},e_{\alpha_{i_{m}}}\right]\right]\right]$
where $\alpha_{i}\in\tilde{\pi}$ and $\alpha=\alpha_{i_{1}}+\cdots+\alpha_{i{}_{m}}$,
$i_{1}\le\ldots\le i_{m}$. For $\alpha\in\tilde{\Delta}^{-}$ define
$f_{\alpha}$ similarly.

If for all $i=1,\ldots,n$, $\sigma\left(e_{\alpha_{i}}\right)$ is
a scalar multiple of $e_{\alpha_{j}}$ for some $\alpha_{j}\in\tilde{\pi}$,
we call $\sigma$ an \emph{almost-diagram} automorphism and denote
$\sigma\left(\alpha_{i}\right):=\alpha_{j}$. For Lie algebras, all
finite order automorphisms are conjugated to almost diagram automorphisms
and all twisted affine Lie algebras can be defined using a diagram
automorphism (with no scalar multiples). We show that $A\left(2k,2l\right)^{\left(4\right)}$
can not be defined using an almost-diagram automorphism.

\subsubsection{$A\left(2k,2l-1\right)^{\left(2\right)}$ }

Let us define an automorphism $\sigma$ of order $2$ on $A\left(2k,2l-1\right)$.
Take 
\[
\tilde{\pi}=\left\{ \varepsilon_{1}-\varepsilon_{2},\ldots,\varepsilon_{2k}-\varepsilon_{2k+1},\varepsilon_{2k+1}-\delta_{1},\delta_{1}-\delta_{2},\ldots,\delta_{2l-1}-\delta_{2l}\right\} .
\]
Then $\sigma$ is defined by
\begin{eqnarray*}
\sigma\left(e_{\varepsilon_{i}-\varepsilon_{i+1}}\right)=e_{\varepsilon_{2k+1-i}-\varepsilon_{2k+2-i}}, & \sigma\left(e_{\delta_{i}-\delta_{i+1}}\right)=e_{\delta_{2l-i}-\delta_{2l+1-i}}, & \sigma\left(e_{\varepsilon_{2k+1}-\delta_{1}}\right)=f_{\varepsilon_{1}-\delta_{2l}},\\
\sigma\left(f_{\varepsilon_{i}-\varepsilon_{i+1}}\right)=f_{\varepsilon_{2k+1-i}-\varepsilon_{2k+2-i}}, & \sigma\left(f_{\delta_{i}-\delta_{i+1}}\right)=f_{\delta_{2l-i}-\delta_{2l+1-i}}, & \sigma\left(f_{\varepsilon_{2k+1}-\delta_{1}}\right)=-e_{\varepsilon_{1}-\delta_{2l}}.
\end{eqnarray*}

\subsubsection{$A\left(2k-1,2l-1\right)^{\left(2\right)}$ }

Let us define an automorphism $\sigma$ of order $2$ on $A\left(2k-1,2l-1\right)$
(i.e. $\mathfrak{psl}\left(2k,2k\right)$ if $k=l$). Take 
\[
\tilde{\pi}=\left\{ \varepsilon_{1}-\varepsilon_{2},\ldots,\varepsilon_{2k-1}-\varepsilon_{2k},\varepsilon_{2k}-\delta_{1},\delta_{1}-\delta_{2},\ldots,\delta_{2l-1}-\delta_{2l}\right\} .
\]
Then $\sigma$ is defined by
\begin{eqnarray*}
\sigma\left(e_{\varepsilon_{i}-\varepsilon_{i+1}}\right)=e_{\varepsilon_{2k-i}-\varepsilon_{2k+1-i}}\cdot\left(-1\right)^{-\delta_{i,k}}, & \sigma\left(e_{\delta_{i}-\delta_{i+1}}\right)=e_{\delta_{2l-i}-\delta_{2l+1-i}}, & \sigma\left(e_{\varepsilon_{2k+1}-\delta_{1}}\right)=f_{\varepsilon_{1}-\delta_{2l}},\\
\sigma\left(f_{\varepsilon_{i}-\varepsilon_{i+1}}\right)=f_{\varepsilon_{2k-i}-\varepsilon_{2k+1-i}}\cdot\left(-1\right)^{-\delta_{i,k}}, & \sigma\left(f_{\delta_{i}-\delta_{i+1}}\right)=f_{\delta_{2l-i}-\delta_{2l+1-i}}, & \sigma\left(f_{\varepsilon_{2k+1}-\delta_{1}}\right)=-e_{\varepsilon_{1}-\delta_{2l}}.
\end{eqnarray*}

\subsubsection{$A\left(2k,2l\right)^{\left(4\right)}$ }

Let us define an automorphism $\sigma$ of order $4$ on $A\left(2k,2l\right)$
(i.e. $\mathfrak{psl}\left(2k+1,2k+1\right)$ if $k=l$). Take 
\[
\tilde{\pi}=\left\{ \varepsilon_{1}-\varepsilon_{2},\ldots,\varepsilon_{2k}-\varepsilon_{2k+1},\varepsilon_{2k+1}-\delta_{1},\delta_{1}-\delta_{2},\ldots,\delta_{2l}-\delta_{2l+1}\right\} .
\]
Then $\sigma$ is defined as for $A\left(2k,2l-1\right)$ on $\left\{ \varepsilon_{1}-\varepsilon_{2},\ldots,\varepsilon_{2k}-\varepsilon_{2k+1},\varepsilon_{2k+1}-\delta_{1},\delta_{1}-\delta_{2},\ldots,\delta_{2l-1}-\delta_{2l}\right\} $
and 
\begin{eqnarray*}
\sigma\left(e_{\delta_{2l}-\delta_{2l+1}}\right) & = & -i\cdot f_{\delta_{1}-\delta_{2l+1}}\\
\sigma\left(f_{\delta_{2l}-\delta_{2l+1}}\right) & = & i\cdot e_{\delta_{1}-\delta_{2l+1}}.
\end{eqnarray*}

\begin{prop}
The algebra $A\left(2k,2l\right)^{\left(4\right)}$ can not be defined
using an almost-diagram automorphism.\end{prop}
\begin{proof}
Let $\mathfrak{h}:=\mathfrak{g}\cap\tilde{\mathfrak{h}}$, where $\mathfrak{g}$
is the algebra formed by the fixed points of $\sigma$ and $\tilde{\mathfrak{h}}$
the Cartan subalgebra of $\tilde{\mathfrak{g}}=A\left(2k,2l\right)$.
Note that $e_{\varepsilon_{k+1}-\delta_{2l+1}}$ and $f_{\varepsilon_{k+1}-\delta_{2l+1}}$
commute with $\mathfrak{h}$. This gives rise to the imaginary odd
root $\delta$ of $\hat{\mathfrak{g}}$ with the root vector $t\otimes\left(e_{\varepsilon_{k+1}-\delta_{2l+1}}+f_{\varepsilon_{k+1}-\delta_{2l+1}}\right)$.
Let us show that this situation can not occur for almost-diagram automorphisms.

Suppose $\sigma$ is an almost diagram automorphism. Then it would
permute the fundamental co-roots $\varpi_{\alpha}$, since if $\left[\varpi_{\alpha},e_{\beta}\right]=\delta_{\alpha,\beta}e_{\beta}$
then $\left[\sigma\left(\varpi_{\alpha}\right),e_{\sigma\left(\beta\right)}\right]=\delta_{\sigma\left(\alpha\right),\sigma\left(\beta\right)}e_{\sigma\left(\beta\right)}$.
Hence $h:=\sum\varpi_{\alpha}$ belongs to $\mathfrak{h}$. However
$h$ is a regular element of $\tilde{\mathfrak{g}}$, that is $\left[h,e_{\gamma}\right]=\mbox{ht}\left(\gamma\right)e_{\gamma}$,
$\left[h,f_{\gamma}\right]=-\mbox{ht}\left(\gamma\right)f_{\gamma}$.
Thus, the centralizer of $\mathfrak{h}$ in $\tilde{\mathfrak{g}}$
is $\tilde{\mathfrak{h}}$ , in particular, there are no imaginary
odd roots.
\end{proof}

\subsubsection{$D\left(k+1,l\right)^{\left(2\right)}$ and $C\left(k+1\right)^{\left(2\right)}$ }

Let us define an automorphism $\sigma$ of order $2$ on $D\left(k+1,l\right)$
and $C\left(k+1\right)$ . For $D\left(k+1,l\right)$ take 
\[
\tilde{\pi}=\left\{ \varepsilon_{1}-\varepsilon_{2},\ldots,\varepsilon_{k}-\varepsilon_{k+1},\varepsilon_{k+1}-\delta_{1},\delta_{1}-\delta_{2},\ldots,\delta_{l-1}-\delta_{l},2\delta_{l}\right\} 
\]
and for $C\left(k+1\right)$ take 
\[
\tilde{\pi}=\left\{ \varepsilon_{1}-\varepsilon_{2},\ldots,\varepsilon_{k}-\varepsilon_{k+1},\varepsilon_{k+1}-\delta_{1}\right\} .
\]
The automorphism $\sigma$ acts by
\begin{eqnarray*}
\sigma\left(e_{\varepsilon_{k}-\varepsilon_{k+1}}\right)=e_{\varepsilon_{k}+\varepsilon_{k+1}}, &  & \sigma\left(e_{\varepsilon_{k+1}-\delta_{1}}\right)=f_{\varepsilon_{k+1}+\delta_{1}},\\
\sigma\left(f_{\varepsilon_{k}-\varepsilon_{k+1}}\right)=f_{\varepsilon_{k}+\varepsilon_{k+1}}, &  & \sigma\left(f_{\varepsilon_{k+1}-\delta_{1}}\right)=e_{\varepsilon_{k+1}-\delta_{1}}.
\end{eqnarray*}
fixing the rest of the Chevalley generators.
\begin{rem}
The automorphism $\sigma$ is a diagram automorphism with respect
to 
\[
\tilde{\pi}=\left\{ \varepsilon_{1}-\varepsilon_{2},\ldots,\varepsilon_{k-1}-\varepsilon_{k},\varepsilon_{k}-\delta_{1},\delta_{1}-\delta_{2},\ldots,\delta_{l-1}-\delta_{l},\delta_{l}-\varepsilon_{k+1},\delta_{l}+\varepsilon_{k+1}\right\} .
\]

\end{rem}

\subsubsection{$G\left(3\right)^{\left(2\right)}$ }

Let us define an automorphism $\sigma$ of order $2$ on $G\left(3\right)$.
Take 
\[
\tilde{\pi}=\left\{ \varepsilon_{3}-\varepsilon_{2},\varepsilon_{2}-\delta_{1},\delta_{1}\right\} .
\]
Then $\sigma$ is defined by
\begin{eqnarray*}
 & \sigma\left(e_{\delta_{1}}\right)=-e_{\delta_{1}},\\
 & \sigma\left(f_{\delta_{1}}\right)=-f_{\delta_{1}}.
\end{eqnarray*}
and fixing the rest of the Chevalley generators.

\subsection{Description of the root systems of the twisted affine Lie superalgebras
and choices of simple roots.\label{sub:Appendix description of simple roots}}

In this section, we describe the root systems of the twisted affine
Lie superalgebras for which we prove the denominator identity, see
Table 1. The root systems are described in terms of a basis $\left\{ \varepsilon_{i},\delta_{j},\delta\mid1\le i\le k,1\le j\le l\right\} $.
The bilinear form $\left(\cdot,\cdot\right)$ defined by 
\[
\left(\varepsilon_{i},\varepsilon_{j}\right)=-\left(\delta_{i},\delta_{j}\right)=\delta_{ij},\quad\left(\varepsilon_{i},\delta_{j}\right)=0
\]
 when $h^{\vee}\ne0$ and $\hat{\mathfrak{g}}\ne G\left(3\right)^{\left(2\right)}$.
When $h^{\vee}=0$ we have 
\[
\left(\varepsilon_{i},\varepsilon_{j}\right)=-\left(\delta_{i},\delta_{j}\right)=-\delta_{ij},\quad\left(\varepsilon_{i},\delta_{j}\right)=0
\]
The root system of $G\left(3\right)^{\left(2\right)}$ is described
by the basis $\left\{ \varepsilon_{1},\varepsilon_{2},\varepsilon_{3}\right\} $
and the inner product is defined such that $\left(\varepsilon_{1},\varepsilon_{1}\right)=1\frac{1}{2}$,
$\left(\varepsilon_{2},\varepsilon_{2}\right)=\frac{1}{2}$, $\left(\varepsilon_{3},\varepsilon_{3}\right)=-2$
and $\left(\varepsilon_{i},\varepsilon_{j}\right)=0$ if $i\ne j$.
Here $\delta$ denotes the minimal imaginary root.

In Table 2 we present a choice of simple roots which is used to prove
Theorem \ref{thm:identity} for each root system of a twisted affine
Lie superalgebra. In Table 3 we list the types of the finite part
and $\hat{\Delta}'$ and $\hat{\Delta}''$ of each algebra.

\begin{flushleft}
\begin{table}[H]
\caption{Root systems}

\raggedright{}%
\begin{tabular}{lll}
\hline 
The algebra &  & Roots\tabularnewline
\hline 
\textcolor{black}{\small $A\left(2k,2l-1\right)^{\left(2\right)}$} & \textcolor{black}{\small $k\ge l$} & {\small $\hat{\Delta}_{\bar{0}}=\left\{ s\delta_{s\ne0},s\delta\pm\varepsilon_{i}\pm\varepsilon_{j},s\delta\pm\varepsilon_{i},\left(2s+1\right)\delta\pm2\varepsilon_{i},s\delta\pm\delta_{g}\pm\delta_{h},2s\delta\pm2\delta_{h}\right\} $}\tabularnewline
 &  & {\small $\hat{\Delta}_{\bar{1}}=\left\{ s\delta\pm\delta_{g}\pm\varepsilon_{i},s\delta\pm\delta_{g}\right\} $}\tabularnewline
\hline 
\textcolor{black}{\small $A\left(2l,2k-1\right)^{\left(2\right)}$} & \textcolor{black}{\small $k\ge l+1$} & {\small $\hat{\Delta}_{\bar{0}}=\left\{ s\delta_{s\ne0},s\delta\pm\delta_{g}\pm\delta_{h},s\delta\pm\delta_{g},\left(2s+1\right)\delta\pm2\delta_{g},s\delta\pm\varepsilon_{i}\pm\varepsilon_{j},2s\delta\pm2\varepsilon_{i}\right\} $}\tabularnewline
 &  & {\small $\hat{\Delta}_{\bar{1}}=\left\{ s\delta\pm\varepsilon_{j}\pm\delta_{g},s\delta\pm\varepsilon_{i}\right\} $}\tabularnewline
\hline 
\textcolor{black}{\small $A\left(2k-1,2l-1\right)^{\left(2\right)}$} & \textcolor{black}{\small $k\ge l+1$} & {\small $\hat{\Delta}_{\bar{0}}=\left\{ s\delta_{s\ne0},s\delta\pm\varepsilon_{i}\pm\varepsilon_{j},s\delta\pm\delta_{g}\pm\delta_{h},2s\delta\pm2\delta_{g},\left(2s+1\right)\delta\pm2\varepsilon_{i}\right\} $}\tabularnewline
 &  & {\small $\hat{\Delta}_{\bar{1}}=\left\{ s\delta\pm\varepsilon_{i}\pm\delta_{g}\right\} $}\tabularnewline
\hline 
\textcolor{black}{\small $A\left(2l-1,2k-1\right)^{\left(2\right)}$} & {\small $k\ge l$} & {\small $\hat{\Delta}_{\bar{0}}=\left\{ s\delta_{s\ne0},s\delta\pm\delta_{g}\pm\delta_{h},s\delta\pm\varepsilon_{i}\pm\varepsilon_{j},2s\delta\pm2\varepsilon_{i},\left(2s+1\right)\delta\pm2\delta_{g}\right\} $}\tabularnewline
 &  & {\small $\hat{\Delta}_{\bar{1}}=\left\{ s\delta\pm\varepsilon_{i}\pm\delta_{g}\right\} $}\tabularnewline
\hline 
\textcolor{black}{\small $A\left(2k,2l\right)^{\left(4\right)}$} & \textcolor{black}{\small $k\ge l+1$} & {\small $\hat{\Delta}_{\bar{0}}=\left\{ 2s\delta_{s\ne0},2s\delta\pm\varepsilon_{i}\pm\varepsilon_{j},2s\delta\pm\varepsilon_{i},\left(4s+2\right)\delta\pm2\varepsilon_{i},\right.$}\tabularnewline
 &  & {\small $\left.2s\delta\pm\delta_{h}\pm\delta_{h},\left(2s+1\right)\delta\pm\delta_{g},4s\delta\pm2\delta_{g}\right\} $}\tabularnewline
 &  & {\small $\hat{\Delta}_{\bar{1}}=\left\{ \left(2s+1\right)\delta,\left(2s+1\right)\delta\pm\varepsilon_{i},2s\delta\pm\delta_{g},2s\delta\pm\varepsilon_{i}\pm\delta_{g}\right\} $}\tabularnewline
\hline 
\textcolor{black}{\small $A\left(2l,2k\right)^{\left(4\right)}$} & \textcolor{black}{\small $k\ge l$} & {\small $\hat{\Delta}_{\bar{0}}=\left\{ 2s\delta_{s\ne0},2s\delta\pm\delta_{g}\pm\delta_{h},2s\delta\pm\delta_{g},\left(4s+2\right)\delta\pm2\delta_{g},\right.$}\tabularnewline
 &  & {\small $\left.2s\delta\pm\varepsilon_{i}\pm\varepsilon_{j},\left(2s+1\right)\delta\pm\varepsilon_{i},4s\delta\pm2\varepsilon_{j}\right\} $}\tabularnewline
 &  & {\small $\hat{\Delta}_{\bar{1}}=\left\{ \left(2s+1\right)\delta,\left(2s+1\right)\delta\pm\delta_{g},2s\delta\pm\varepsilon_{i},2s\delta\pm\delta_{g}\pm\varepsilon_{i}\right\} $}\tabularnewline
\hline 
\textcolor{black}{\small $C\left(l+1\right)^{\left(2\right)}$, $D\left(k+1,l\right)^{\left(2\right)}$} & \textcolor{black}{\small $k\ge l+1$} & {\small $\hat{\Delta}_{\bar{0}}=\left\{ s\delta_{s\ne0},2s\delta\pm\varepsilon_{i}\pm\varepsilon_{j},2s\delta\pm\delta_{g}\pm\delta_{h},2s\delta\pm2\delta_{g},s\delta\pm\varepsilon_{i}\right\} $}\tabularnewline
 &  & {\small $\hat{\Delta}_{\bar{1}}=\left\{ 2s\delta\pm\varepsilon_{i}\pm\delta_{g},s\delta\pm\delta_{g}\right\} $}\tabularnewline
\hline 
\textcolor{black}{\small $C\left(l+1\right)^{\left(2\right)}$, $D\left(l+1,k\right)^{\left(2\right)}$} & \textcolor{black}{\small $k\ge l$} & {\small $\hat{\Delta}_{\bar{0}}=\left\{ s\delta_{s\ne0},2s\delta\pm\delta_{g}\pm\delta_{h},2s\delta\pm\varepsilon_{i}\pm\varepsilon_{j},2s\delta\pm2\varepsilon_{i},s\delta\pm\delta_{g}\right\} $}\tabularnewline
 &  & {\small $\hat{\Delta}_{\bar{1}}=\left\{ 2s\delta\pm\delta_{g}\pm\varepsilon_{i},s\delta\pm\varepsilon_{i}\right\} $}\tabularnewline
\hline 
\textcolor{black}{\small $G\left(3\right)^{\left(2\right)}$} &  & {\small $\hat{\Delta}_{\bar{0}}=\left\{ 2s\delta_{s\ne0},2s\delta\pm2\varepsilon_{1},2s\delta\pm2\varepsilon_{2},2s\delta\pm2\varepsilon_{3},\left(2s+1\right)\delta\pm\left(3\varepsilon_{2}+\varepsilon_{1}\right),\right.$}\tabularnewline
 &  & $\left.\left(2s+1\right)\delta\pm\left(3\varepsilon_{2}-\varepsilon_{1}\right),\left(2s+1\right)\delta\pm\left(\varepsilon_{1}+\varepsilon_{2}\right),\left(2s+1\right)\delta\pm\left(\varepsilon_{2}-\varepsilon_{1}\right)\right\} $\tabularnewline
 &  & {\small $\hat{\Delta}_{\bar{1}}=\left\{ {\color{black}{\color{green}{\color{black}2s\delta\pm\varepsilon_{1}\pm\varepsilon_{2}\pm}{\color{black}\varepsilon_{3}}},{\color{magenta}{\color{black}\left({\color{black}2s+1}\right)}{\color{black}\delta}{\color{black}\pm2\varepsilon_{2}\pm}{\color{black}\varepsilon_{{\color{black}3}}}},{\color{blue}{\color{black}\left(2s+1\right)\delta\pm\varepsilon_{3}}}}\right\} $}\tabularnewline
\hline 
\end{tabular}
\end{table}

\par\end{flushleft}

\begin{table}[H]
\caption{Root system types}
\textbf{}%
\begin{tabular}{cccc}
\hline 
\textbf{The algebra} & \textbf{Finite Part} & \textbf{$\hat{\Delta}'$} & \textbf{$\hat{\Delta}''$}\tabularnewline
\hline 
\textbf{$A\left(2k,2l-1\right)^{\left(2\right)}$, $k\ge l$} & \textbf{$B\left(k,l\right)$ } & \textbf{$A_{2k}^{\left(2\right)}$} & \textbf{$A_{2l-1}^{\left(2\right)}$}\tabularnewline
\hline 
\textbf{$A\left(2l,2k-1\right)^{\left(2\right)}$, $k\ge l+1$} & \textbf{$B\left(l,k\right)$} & \textbf{$A_{2k-1}^{\left(2\right)}$} & \textbf{$A_{2l}^{\left(2\right)}$}\tabularnewline
\hline 
\textbf{$A\left(2k-1,2l-1\right)^{\left(2\right)}$, $k\ge l$.} & \textbf{$D\left(k,l\right)$} & \textbf{$A_{2k-1}^{\left(2\right)}$} & \textbf{$A_{2l-1}^{\left(2\right)}$}\tabularnewline
\hline 
\textbf{$A\left(2l-1,2k-1\right)^{\left(2\right)}$, $k\ge l$.} & \textbf{$D\left(l,k\right)$} & \textbf{$A_{2k-1}^{\left(2\right)}$} & \textbf{$A_{2l-1}^{\left(2\right)}$}\tabularnewline
\hline 
\textbf{$A\left(2k,2l\right)^{\left(4\right)}$, $k\ge l+1$} & \textbf{$B\left(k,l\right)$} & \textbf{$A_{2k}^{\left(2\right)}$} & \textbf{$A_{2l}^{\left(2\right)}$}\tabularnewline
\hline 
\textbf{$A\left(2l,2k\right)^{\left(4\right)}$, $k\ge l$} & \textbf{$B\left(l,k\right)$} & \textbf{$A_{2k}^{\left(2\right)}$} & \textbf{$A_{2l}^{\left(2\right)}$}\tabularnewline
\hline 
\textbf{$D\left(k+1,l\right)^{\left(2\right)}$, $k\ge l$} & \textbf{$B\left(k,l\right)$} & \textbf{$D_{k+1}^{\left(2\right)}$} & \textbf{$C_{l}^{\left(1\right)}$}\tabularnewline
\hline 
\textbf{$D\left(l+1,k\right)^{\left(2\right)}$, $k\ge l+1$} & \textbf{$B\left(l,k\right)$} & \textbf{$C_{k}^{\left(1\right)}$} & \textbf{$D_{l+1}^{\left(2\right)}$}\tabularnewline
\hline 
\textbf{$G\left(3\right)^{\left(2\right)}$} & \textbf{$D\left(1,2,-\frac{3}{4}\right)$} & \textbf{$G_{2}^{\left(1\right)}$} & \textbf{$A_{1}^{\left(1\right)}$}\tabularnewline
\hline 
\end{tabular}
\end{table}
\begin{table}[H]
\caption{Choices of simple roots}
\begin{tabular}{lll}
\hline 
The algebra &  & Choice of simple roots\tabularnewline
\hline 
\textcolor{black}{\small $A\left(2k,2l-1\right)^{\left(2\right)}$} & \textcolor{black}{\small $k\ge l+1$} & {\small $\varepsilon_{1}-\delta_{1},\delta_{1}-\varepsilon_{2},\ldots,\varepsilon_{l}-\delta_{l},\delta_{l}-\varepsilon_{l+1},$}\tabularnewline
 &  & {\small $\varepsilon_{l+1}-\varepsilon_{l+2},\ldots,\varepsilon_{k-1}-\varepsilon_{k},\varepsilon_{k},\delta-2\varepsilon_{1}$}\tabularnewline
\hline 
\textcolor{black}{\small $A\left(2k,2k-1\right)^{\left(2\right)}$} &  & {\small $\delta_{1}-\varepsilon_{1},\ldots,\delta_{k}-\varepsilon_{k},\varepsilon_{k},\delta-\varepsilon_{1}-\delta_{1}$}\tabularnewline
\hline 
\textcolor{black}{\small $A\left(2l,2l+1\right)^{\left(2\right)}$} &  & {\small $\varepsilon_{1}-\delta_{1},\ldots,\varepsilon_{l}-\delta_{l},\delta_{l}-\varepsilon_{l+1},\varepsilon_{l+1},\delta-\delta_{1}-\varepsilon_{1}$}\tabularnewline
\hline 
\textcolor{black}{\small $A\left(2l,2k-1\right)^{\left(2\right)}$} & \textcolor{black}{\small $k\ge l+2$} & {\small $\varepsilon_{1}-\varepsilon_{2},\varepsilon_{2}-\delta_{1},\delta_{1}-\varepsilon_{3},\ldots,\varepsilon_{l+1}-\delta_{l},\delta_{l}-\varepsilon_{l+2},$}\tabularnewline
 &  & {\small $\varepsilon_{l+2}-\varepsilon_{l+3},\ldots,\varepsilon_{k-1}-\varepsilon_{k},\varepsilon_{k},\delta-\varepsilon_{1}-\varepsilon_{2}$}\tabularnewline
\hline 
\textcolor{black}{\small $A\left(2k-1,2l-1\right)^{\left(2\right)}$} & \textcolor{black}{\small $k\ge l+2$} & {\small $\varepsilon_{1}-\delta_{1},\delta_{1}-\varepsilon_{2},\ldots,\varepsilon_{l-1}-\delta_{l-1},\delta_{l}-\varepsilon_{l+1},$}\tabularnewline
 &  & {\small $\varepsilon_{l+1}-\varepsilon_{l+2},\ldots\varepsilon_{k-1}-\varepsilon_{k},\varepsilon_{k-1}+\varepsilon_{k},\delta-2\varepsilon_{1}$}\tabularnewline
\hline 
\textcolor{black}{\small $A\left(2l+1,2l-1\right)^{\left(2\right)}$} &  & {\small $\varepsilon_{1}-\delta_{1},\delta_{1}-\varepsilon_{2},\ldots,\varepsilon_{l-1}-\delta_{l-1},\delta_{l}-\varepsilon_{l+1},\delta_{l}+\varepsilon_{l+1},\delta-2\varepsilon_{1}$}\tabularnewline
\hline 
\textcolor{black}{\small $A\left(2k-1,2k-1\right)^{\left(2\right)}$} &  & {\small $\varepsilon_{1}-\delta_{1},\delta_{1}-\varepsilon_{2},\ldots,\varepsilon_{k}-\delta_{k},\varepsilon_{k}+\delta_{k},\delta-\varepsilon_{1}-\delta_{1}$}\tabularnewline
\hline 
\textcolor{black}{\small $A\left(2l-1,2l+1\right)^{\left(2\right)}$} &  & {\small $\varepsilon_{1}-\varepsilon_{2},\varepsilon_{2}-\delta_{1},\delta_{1}-\varepsilon_{3},\ldots,\varepsilon_{l+1}-\delta_{l},\varepsilon_{l+1}+\delta_{l},\delta-\varepsilon_{1}-\varepsilon_{2}$}\tabularnewline
\hline 
\textcolor{black}{\small $A\left(2l-1,2k-1\right)^{\left(2\right)}$} & \textcolor{black}{\small $k\ge l+2$} & {\small $\varepsilon_{1}-\varepsilon_{2},\varepsilon_{2}-\delta_{1},\delta_{1}-\varepsilon_{3},\ldots,\varepsilon_{l+1}-\delta_{l},\delta_{l}-\varepsilon_{l+2},$}\tabularnewline
 &  & {\small $\varepsilon_{l+2}-\varepsilon_{l+3},\ldots\varepsilon_{k-1}-\varepsilon_{k},2\varepsilon_{k},\delta-\varepsilon_{1}-\varepsilon_{2}$}\tabularnewline
\hline 
\textcolor{black}{\small $A\left(2k,2l\right)^{\left(4\right)}$} & \textcolor{black}{\small $k\ge l+1$} & {\small $\varepsilon_{1}-\delta_{1},\delta_{1}-\varepsilon_{2},\ldots,\varepsilon_{l-1}-\delta_{l},\delta_{l}-\varepsilon_{l},$}\tabularnewline
 &  & {\small $\varepsilon_{l}-\varepsilon_{l+1},\ldots\varepsilon_{k-1}-\varepsilon_{k},\varepsilon_{k},\delta-\varepsilon_{1}$}\tabularnewline
\hline 
\textcolor{black}{\small $A\left(2k,2k\right)^{\left(4\right)}$} &  & {\small $\varepsilon_{1}-\delta_{1},\delta_{1}-\varepsilon_{2},\ldots,\varepsilon_{k}-\delta_{k},\delta_{k},\delta-\varepsilon_{1}$}\tabularnewline
\hline 
\textcolor{black}{\small $A\left(2l,2k\right)^{\left(4\right)}$} & \textcolor{black}{\small $k\ge l+1$} & {\small $\varepsilon_{1}-\delta_{1},\delta_{1}-\varepsilon_{2},\ldots,\varepsilon_{l}-\delta_{l},\delta_{l}-\varepsilon_{l+1},$}\tabularnewline
 &  & {\small $\varepsilon_{l}-\varepsilon_{l+1},\ldots\varepsilon_{k-1}-\varepsilon_{k},\varepsilon_{k},\delta-\varepsilon_{1}$}\tabularnewline
\hline 
\textcolor{black}{\small $C\left(l+1\right)^{\left(2\right)}$, $D\left(k+1,l\right)^{\left(2\right)}$} & \textcolor{black}{\small $k\ge l+1$} & {\small $\varepsilon_{1}-\delta_{1},\delta_{1}-\varepsilon_{2},\ldots,\varepsilon_{l}-\delta_{l},\delta_{l}-\varepsilon_{l+1},$}\tabularnewline
 &  & {\small $\varepsilon_{l+1}-\varepsilon_{l+2},\ldots,\varepsilon_{k-1}-\varepsilon_{k},\varepsilon_{k},\delta-\varepsilon_{1}$}\tabularnewline
\hline 
\textcolor{black}{\small $D\left(k+1,k\right)^{\left(2\right)}$} &  & {\small $\varepsilon_{1}-\delta_{1},\delta_{1}-\varepsilon_{2},\ldots,\varepsilon_{k}-\delta_{k},\delta_{k},\delta-\varepsilon_{1}$}\tabularnewline
\hline 
\textcolor{black}{\small $C\left(l+1\right)^{\left(2\right)}$, $D\left(l+1,k\right)^{\left(2\right)}$} & \textcolor{black}{\small $k\ge l+1$} & {\small $\varepsilon_{1}-\delta_{1},\delta_{1}-\varepsilon_{2},\ldots,\varepsilon_{l}-\delta_{l},\delta_{l}-\varepsilon_{l+1},$}\tabularnewline
 &  & {\small $\varepsilon_{l+1}-\varepsilon_{l+2},\ldots,\varepsilon_{k-1}-\varepsilon_{k},\varepsilon_{k},\delta-\varepsilon_{1}$}\tabularnewline
\hline 
\textcolor{black}{\small $G\left(3\right)^{\left(2\right)}$} &  & {\small $\varepsilon_{3}-\varepsilon_{2}-\varepsilon_{1},2\varepsilon_{1},2\varepsilon_{2},\delta-\left(\varepsilon_{3}+2\varepsilon_{2}\right)$}\tabularnewline
\hline 
\end{tabular}
\end{table}

\end{document}